\documentclass{amsart}
\usepackage{mathrsfs}
\usepackage{amssymb}
\usepackage{color}
\usepackage{enumerate}

\advance\topmargin by -\headheight
\advance\topmargin by -\headsep
\evensidemargin=\oddsidemargin
\newtheorem{thm}{Theorem}
\newtheorem{lem}[thm]{Lemma}
\newtheorem{defi}[thm]{Definition}
\newtheorem{cor}[thm]{Corollary}
\newtheorem{prop}[thm]{Proposition}
\newtheorem{compo}{Computation}

\newcommand{\A}{\mathcal{A}}

\begin{document}

\title{Universal measurability and the Hochschild class of the Chern character}

\author{A.L.  Carey}
\address{Mathematical Sciences Institute, Australian National University, Canberra, 2602, Australia, and\newline School of Mathematics and Applied Statistics, University of Wollongong, 2522, Australia}\email{alan.carey@anu.edu.au}
\author{A. Rennie}
\address{School of Mathematics and Applied Statistics, University of Wollongong, 2522, Australia}\email{renniea@uow.edu.au}
\author{F. Sukochev}
\address{School of Mathematics and Statistics, University of New South Wales, Sydney, 2052, Australia.}
\email{f.sukochev@unsw.edu.au, d.zanin@unsw.edu.au}
\author{D. Zanin}

\subjclass[2000]{}


\date{}

\dedicatory{}

\begin{abstract} 
We study notions of measurability for singular traces, and characterise universal
measurability for operators in Dixmier ideals.
This measurability result is then applied to improve on the various proofs
of Connes' identification of the
Hochschild class of the Chern character of Dixmier summable spectral triples.

The measurability results show that the identification of the Hochschild class is 
independent of the choice of singular trace. As a corollary we obtain strong information on the 
asymptotics of the eigenvalues
of operators naturally associated to spectral triples $(\A,H,D)$ and Hochschild cycles for $\A$. 
\end{abstract}

%
%

\maketitle

\tableofcontents

\bibliographystyle{plain}

\parindent=0.0pt
\parskip=3.0pt

\section{Introduction}

 The Character Theorem of Alain Connes, \cite[Theorem 8, IV.2.$\gamma$]{Connes}, 
 has been the subject of a number of papers (see  \cite{BeF, CPRS, GVF}). Our 
 main results in this paper (Theorem \ref{main result} and Corollary \ref{main cor}) 
 extend and strengthen this theorem. In essence, this result identifies the Hochschild cohomology class
 of the cyclic cohomology Chern character of a spectral triple. To do this, one must suppose
 that the spectral triple has `integral spectral dimension', say $p\geq 1$, and that 
 the spectral triple is $(p,\infty)$-summable.
 
 The definition of $(p,\infty)$-summability involves one of two ideals, denoted here by
 $\mathcal{L}_{1,\infty}$ and $\mathcal{M}_{1,\infty}$, or the related ideals 
 $\mathcal{L}_{p,\infty}$ and $\mathcal{M}_{1,\infty}^{(p)}$. This is where potential confusion can arise,
 as well as much difficulty since the ideal $\mathcal{M}_{1,\infty}$ is more subtle than 
 $\mathcal{L}_{1,\infty}$. The key technical improvement in this paper is the identification of 
 a criterion guaranteeing measurability with respect to families of traces on these ideals.

We remark that we began this investigation because there is a gap in the proof of Lemma 14
in \cite{CPRS} for the case $p=1$ and the ideal denoted (and defined) below by $\mathcal{M}_{1,\infty}$.
Rather than simply produce an erratum we decided to revisit the whole argument in the 
light of progress made in the last 10 years \cite{CRSS,CarSuk, DFWW,KLPS} which 
provides, amongst other contributions, a more powerful algebraic approach to these 
issues resulting in two advances.

(i) We prove Connes' result for arbitrary traces on $\mathcal{L}_{1,\infty}$ (other
proofs hold only for the original trace discovered by  Dixmier). This has  
interesting  implications for the eigenvalues of the Hochschild cycles.

(ii)  We prove the analogue of this theorem for the (Macaev-Dixmier) 
ideal $\mathcal{M}_{1,\infty}$ as well as the $p$-convexifications $\mathcal{M}_{1,\infty}^{(p)}$
(introduced in \cite{CRSS}, and denoted there by ${\mathcal Z}_p$). 
The latter ideal strictly contains $\mathcal{L}_{p,\infty}$. Our proof holds for  a wide 
class of traces on $\mathcal{M}_{1,\infty}$.

Moreover we make an interesting technical innovation in this  current approach by 
exploiting recently discovered connections\footnote{For a detailed exposition of the 
connections, we refer the reader to \cite{LSZ}. } between Dixmier traces and heat 
kernel functionals exposed in \cite{SZ-JFA}. These connections result in a  
streamlining of the proof and a major reduction in the  number of estimates 
needed (compared to the proof in \cite{CPRS}). 

Our results are presented in the context of operator ideals in $\mathcal{L}(H)$ for a
separable infinite dimensional Hilbert space. All of our results carry over to the general case
of operator ideals of a semifinite von Neumann algebra although we do not 
present the argument in that
generality here. We have simplified
our approach, compared to \cite{CPRS}, by assuming that our 
spectral triples (introduced in Section \ref{sec:prelim-NCG}) are smooth, however,
by taking more care in Lemma \ref{f second est}
we can recover the minimal smoothness requirements
of \cite[Lemma 2]{CPRS}.

As an indication of the improvements we have obtained, we state a 
consequence of our results which is applicable
to numerous examples in the literature, including the case of 
Dirac operators on compact manifolds
and the noncommutative torus.

{\bf Theorem}.
{\it Let $(\A,H,D)$ be a spectral triple with $(1+D^2)^{-1/2}\in\mathcal{L}_{p,\infty}$, where $p$
is an integer of the same parity as the spectral triple. If the spectral triple is even we let $\Gamma$
be the grading, and otherwise let $\Gamma=1$.
For every Hochschild cycle $c\in\mathcal{\A}^{\otimes p+1}$, 
$
c=\sum_ic^i_0\otimes c^i_1\otimes\cdots\otimes c^p_i
$
set $\Omega(c)=\sum_i\Gamma c^i_0[D,c^i_1]\cdots[D,c^i_p]$. Then  denoting the
(suitably ordered)\footnote{The eigenvalues are counted with 
algebraic multiplicities and arranged so that their absolute values are non-increasing.} 
eigenvalues of $\Omega(c)(1+D^2)^{-p/2}$ by $\lambda_k$ we have
$$
\sum_{k=0}^n\lambda_k= Ch(c)\log(n)+O(1),
$$
where $Ch$ is the Chern character of the $K$-homology class of $(\mathcal{A},H,D)$. In particular,
$\Omega(c)(1+D^2)^{-p/2}$ is universally measurable in the sense of  Definition \ref{def:uni-meas}.}

The necessary background on operator ideals, traces and measurability is presented
in Section \ref{sec:prelim-traces}, and a key abstract measurability criterion is established
in subsection \ref{abstract}. Section \ref{sec:prelim-NCG} summarises what we need about
spectral triples, Chern characters and Hochschild cohomology. We state our main results, Theorem
\ref{main result} and Corollary \ref{main cor}
together with an outline of the proof  in subsection \ref{plan}. Section \ref{sec:proofs}
presents the proofs. An appendix shows how certain 
Hochschild coboundaries are
computed.

{\bf Acknowledgements.} All authors were supported by the Australian Research Council.
AC also acknowledges the Alexander von Humboldt Stiftung and thanks colleagues at the 
University of M\"unster for support while this research was undertaken.

\section{Preliminaries on operator ideals, traces and measurability}
\label{sec:prelim-traces}

\subsection{General notation}
Fix throughout a separable infinite dimensional Hilbert space $H$.
We let $\mathcal{L}(H)$ denote the algebra of all bounded operators on $H.$ 
For a compact operator $T$ on $H$, let $\lambda(k,T)$ and $\mu(k,T)$ 
denote its $k$-th eigenvalue\footnotemark[\value{footnote}]
and $k$-th largest singular value (these are the eigenvalues of $|T|$). The 
sequence $\mu(T)=\{\mu(k,T)\}_{k\geq0}$ is referred to as  the singular 
value sequence of the operator $T$. The standard trace on $\mathcal{L}(H)$ is denoted by 
${\rm Tr}$. For an arbitrary operator $0\leq T\in \mathcal{L}(H)$, we set
$$
n_T(t):={\rm Tr}(E_T(t,\infty)),\ t>0,
$$
where 
$E_{T}{(a,b)}$ stands for the  spectral projection of a self-adjoint operator $T$ 
corresponding to the interval $(a,b)$.
Fix an orthonormal basis in $H$ (the particular choice of a basis is inessential). 
We identify the  algebra $l_{\infty}$ of bounded sequences with the subalgebra 
of all diagonal operators with respect to the chosen basis. For a given sequence 
$\alpha\in l_{\infty},$ we denote the corresponding diagonal operator by ${\rm diag}(\alpha).$


\subsection{Principal ideals $\mathcal{L}_{p,\infty}$ and the Macaev-Dixmier ideal $\mathcal{M}_{1,\infty}$}

For a given $0<p\leq\infty,$ we let $\mathcal{L}_{p,\infty}$ denote the principal ideal in $\mathcal{L}(H)$ generated by the operator ${\rm diag}(\{(k+1)^{-1/p}\}_{k\geq0}).$ Equivalently,
$$
\mathcal{L}_{p,\infty}=\{T\in\mathcal{L}(H): \mu(k,T)=O((k+1)^{-1/p})\}.
$$
These ideals, for different $p$, all admit an equivalent description in terms of 
spectral projections, namely
\begin{equation}
\label{tlp}
T\in\mathcal{L}_{p,\infty}\Longleftrightarrow {\rm Tr}(E_{|T|}(1/n,\infty))=O(n^p).
\end{equation}
We also have
$$
|T|^p\in\mathcal{L}_{1,\infty}\Longleftrightarrow \mu^p(k,T)
=O((k+1)^{-1})\Longleftrightarrow T\in\mathcal{L}_{p,\infty}.
$$
We equip the ideal $\mathcal{L}_{p,\infty},$ $0<p\leq\infty$, 
with a quasi-norm\footnote{A quasinorm satisfies the norm axioms, 
except that the triangle inequality is replaced by $||x+y||\leq K(||x||+||y||)$ 
for some uniform constant $K>1$.}
$$
\|T\|_{p,\infty}=\sup_{k\geq0}(k+1)^{1/p}\mu(k,T),\quad T\in\mathcal{L}_{p,\infty}.
$$

The following H\"older property is widely used throughout the paper:
\begin{equation}
\label{lpi mult}
A_k\in\mathcal{L}_{p_m,\infty},\ 1\leq m\leq n,\Longrightarrow 
\prod_{m=1}^nA_m\in\mathcal{L}_{p,\infty},\ \frac1p=\sum_{m=1}^n\frac1{p_m}.
\end{equation}

We also need the Macaev-Dixmier  ideal $\mathcal{M}_{1,\infty}$, 
also known as a Lorentz space, given by
$$
\mathcal{M}_{1,\infty}=\{A\in\mathcal{L}(H):\ \sup_{n\geq0}\frac1{\log(2+n)}\sum_{k=0}^n\mu(k,A)<\infty\}.
$$
The ideal $\mathcal{M}_{1,\infty}^{(p)}$ initially considered in \cite{CRSS} is the 
$p$-convexification of $\mathcal{M}_{1,\infty}$ defined as follows.
$$
\mathcal{M}_{1,\infty}^{(p)}=\{A\in\mathcal{L}(H):\ |A|^p\in\mathcal{M}_{1,\infty}\}.
$$
The ideal $\mathcal{M}_{1,\infty}^{(p)}$ strictly contains 
$\mathcal{L}_{p,\infty}.$ We refer the reader to the book 
\cite{LSZ} for a detailed discussion of the ideals $\mathcal{L}_{1,\infty}$ and $\mathcal{M}_{1,\infty}$.

\subsection{Traces on $\mathcal{L}_{1,\infty}$}

\begin{defi}\label{trace def} If $\mathcal{I}$ is an 
ideal in $\mathcal{L}(H),$ then a unitarily invariant 
linear functional $\varphi:\mathcal{I}\to\mathbb{C}$ is said to be a trace.
\end{defi}
Since $U^{-1}TU-T=[U^{-1},TU]$ for all $T\in\mathcal{I}$ and 
for all unitaries $U\in\mathcal{L}(H),$ and since the unitaries 
span $\mathcal{L}(H),$ it follows that traces are precisely 
the linear functionals on $\mathcal{I}$ satisfying the condition
$$\varphi(TS)=\varphi(ST),\quad T\in\mathcal{I}, S\in\mathcal{L}(H).$$
The latter may be reinterpreted as the vanishing of the 
linear functional $\varphi$ on the commutator 
subspace\footnote{The commutator subspace of the ideal is, 
in general, not an ideal in $\mathcal{L}(H).$ For example, it 
follows from Theorem \ref{komm description} below that
$$
{\rm diag}\Big(\Big\{\frac{(-1)^k}{k+1}\Big\}_{k\geq0}\Big)\in[\mathcal{L}_{1,\infty},\mathcal{L}(H)],
\quad {\rm diag}\Big(\Big\{\frac1{k+1}\Big\}_{k\geq0}\Big)\notin[\mathcal{L}_{1,\infty},\mathcal{L}(H)].
$$
However, the commutator subspace of the ideal $\mathcal{L}_{1,\infty}$ is an 
ideal in $\mathcal{L}_{1,\infty}$ (as opposed to $\mathcal{L}(H)$). We refer the 
reader to \cite{PW} for the study of such subideals.}
which is denoted $[\mathcal{I},\mathcal{L}(H)]$ and defined to be the linear 
span of all commutators $[T,S]:\ T\in\mathcal{I},$ $S\in\mathcal{L}(H).$ 
It is shown in \cite[Lemma 5.2.2]{LSZ} that  
$\varphi(T_1)=\varphi(T_2)$ whenever 
$0\leq T_1,T_2\in\mathcal{I}$ are such that the singular value 
sequences $\mu(T_1)$ and $\mu(T_2)$ coincide.
For $p>1,$ the ideal $\mathcal{L}_{p,\infty}$ does not admit a 
non-zero trace while for $p=1,$ 
there exists a plethora of traces on $\mathcal{L}_{1,\infty}$ 
(see e.g. \cite{DFWW} or \cite{LSZ}). An example of  a trace on 
$\mathcal{L}_{1,\infty}$ is the restriction (from $\mathcal{M}_{1,\infty}$) 
of the  Dixmier trace introduced in \cite{Dixmier}
that we now explain.

\begin{defi}
The dilation semigroup on $l_\infty $ is defined by setting
$$
\sigma_k(x_0,x_1,\cdots)=(\underbrace{x_0,\cdots,x_0}_{\mbox{$k$ times}},\underbrace{x_1,\cdots,x_1}_{\mbox{$k$ times}},\cdots)
$$
for every $k\geq 1.$
In this paper  {\it a dilation invariant extended limit} means a  
state on the algebra $l_{\infty}$
invariant under $\sigma_k$, $k=2,3,\ldots$.
\end{defi}
\noindent{\bf Example}. Let $\omega$ be a dilation invariant extended limit.
Then the functional ${\rm Tr}_{\omega}:\mathcal{M}_{1,\infty}^+\to\mathbb{C}$ defined by setting
$$
{\rm Tr}_{\omega}(A)=\omega\Big(\Big\{\frac1{\log(2+n)}\sum_{k=0}^n\mu(k,A)\Big\}_{n\geq0}\Big),
\quad 0\leq A\in\mathcal{M}_{1,\infty},
$$
is additive and, therefore, extends to a trace on $\mathcal{M}_{1,\infty}.$ 
We call such traces  {\it Dixmier traces}. 
These traces clearly depend on the choice of the functional $\omega$ on $l_\infty$.
Using a slightly different definition,  this notion of trace was 
applied by Connes \cite{Connes} in noncommutative geometry. 
We also remark that the assumption used by Dixmier of translation invariance for
the functional $\omega$ is redundant  \cite[Theorem 6.3.6]{LSZ}.
%
An extensive discussion of traces, and more recent developments in the theory,
may be found in \cite{LSZ} including a discussion of the following facts.
\begin{enumerate}
\item All Dixmier traces on $\mathcal{L}_{1,\infty}$ are positive.
\item All positive traces on $\mathcal{L}_{1,\infty}$ are continuous in the quasi-norm topology.
\item There exist positive traces on $\mathcal{L}_{1,\infty}$ which are not (restrictions to $\mathcal{L}_{1,\infty}$ from $\mathcal{M}_{1,\infty}$ of)  Dixmier traces 
(see \cite{SSUZ-pietsch}).
\item There exist traces on $\mathcal{L}_{1,\infty}$ which fail to be continuous (see \cite{DFWW}).
\end{enumerate}

We are mostly interested in {\it normalised traces} $\varphi:\mathcal{L}_{1,\infty}\to\mathbb{C},$ that is, satisfying $\varphi(T)=1$ whenever $0\leq T$ is such that $\mu(k,T)=1/(k+1)$ for all $k\geq0.$ We do not require continuity of a normalised trace. 

The following definition, extending that originally introduced in \cite{Connes}, plays an important role here.
\begin{defi} 
\label{def:uni-meas}
An operator $T\in\mathcal{L}_{1,\infty}$ will be said to be
 universally measurable if all normalised traces take the same value on $T.$ 
\end{defi}

The following lemma characterises the universally measurable operators.

\begin{lem} All normalised traces on $\mathcal{L}_{1,\infty}$ take the value $z\in\mathbb C$ on the operator $T$ if and only if
$$
T-z\cdot{\rm diag}\Big(\Big\{\frac1{k+1}\Big\}_{k\geq0}\Big)\in[\mathcal{L}_{1,\infty},\mathcal{L}(H)].
$$
\end{lem}
\begin{proof} Suppose that all normalised traces on $\mathcal{L}_{1,\infty}$ 
take the value $z$ on the operator $T$. For brevity we write 
$T_0={\rm diag}(\{\frac1{k+1}\}_{k\geq0})$. If $T-zT_0$ is not in the commutator 
subspace, then it follows from Zorn's lemma that there exists a linear functional $
\varphi$ on $\mathcal{L}_{1,\infty}$ such that $\varphi|_{[\mathcal{L}_{1,\infty},\mathcal{L}(H)]}=0$ 
and such that $\varphi(T-zT_0)=1$. By Definition \ref{trace def}, $\varphi$ is a trace.
Fix a normalised trace $\varphi_0$. The normalised trace 
$\varphi+(1-\varphi(T_0))\varphi_0$ takes the value $z+1$ at $T$, 
which contradicts the assumption. This proves that $T-zT_0\in[\mathcal{L}_{1,\infty},\mathcal{L}(H)]$. 
The converse assertion follows from the definitions.
\end{proof}

The description of the commutator subspace initially appeared in \cite{Kalton1998} in a very general situation. The statement below appeared first in \cite{KLPS} and for a detailed proof  we refer the reader to Theorem 5.7.6 and Theorem 10.1.3 in \cite{LSZ}.
\begin{prop}\label{komm description} An operator $T\in\mathcal{L}_{1,\infty}$ is universally measurable if and only if
$$\sum_{k=0}^n\lambda(k,T)=z\log(n+1)+O(1),\quad n\geq0,$$
for some constant $z\in\mathbb{C}.$ In this case, $\varphi(T)=z$ for every normalised trace $\varphi.$ In particular
$$
T\in[\mathcal{L}_{1,\infty},\mathcal{L}(H)]\Longleftrightarrow \sum_{k=0}^n\lambda(k,T)
=O(1),\quad n\geq0.
$$
\end{prop}

\subsection{A universal measurability result}
\label{abstract}

In this subsection, we prove a measurability criterion for operators of the 
form $AV,$ $A\in\mathcal{L}(H),$ $V\in\mathcal{L}_{1,\infty},$ or $V\in\mathcal{M}_{1,\infty}.$
This result links measurability with the heat semigroup, thus significantly improving the main result of \cite{CarSuk}. More information on these links can be found in \cite{LSZ} (see also \cite{SZ-JFA}).
The precise statement of our measurability criterion is as follows.

\begin{prop}\label{measurability criterion} Let $0\leq V\in\mathcal{L}(H),$ 
$A\in\mathcal{L}(H)$ and $\alpha>1$ be such that
\begin{equation}
\label{meas eq crit}
{\rm Tr}(AVe^{-(nV)^{-\alpha}})=z\log(n)+O(1),\quad n\to\infty.
\end{equation}
\begin{enumerate}[{\rm (a)}]
\item\label{mesa} If $V\in\mathcal{L}_{1,\infty},$ then $\varphi(AV)=z$ 
for every normalised trace $\varphi$ on $\mathcal{L}_{1,\infty}.$
\item\label{mesb} If $V\in\mathcal{M}_{1,\infty},$ then ${\rm Tr}_{\omega}(AV)=z$ 
for every Dixmier trace ${\rm Tr}_{\omega}$ on $\mathcal{M}_{1,\infty}.$
\end{enumerate}
\end{prop}

We require several Lemmas before presenting the proof of Proposition \ref{measurability criterion}.

\begin{lem}\label{estimate} If $0\leq V\in\mathcal{L}_{1,\infty},$ then, for every $\alpha>1,$ we have
$${\rm Tr}(V^{\alpha}(1-e^{-(nV)^{-\alpha}}))=O(n^{1-\alpha}),\quad{\rm Tr}(e^{-(nV)^{-\alpha}}))=O(n),\quad n\to\infty.$$
\end{lem}
\begin{proof} By the assumption, we have $\mu(k,V)\leq\|V\|_{1,\infty}/(k+1)$ for all 
$k\geq0.$ Select $W\geq V$ (with the same eigenbasis) such that 
$\mu(k,W)=\|V\|_{1,\infty}/(k+1)$ for all $k\geq0.$ An elementary computation shows that the mapping
$g:x\to x(1-e^{-x^{-1}}),\quad x\geq0$,
is increasing. Since $V$ and $W$ commute, $(nV)^{\alpha}\leq (nW)^{\alpha}$ 
for all $n\geq1$ and it follows that $g((nV)^{\alpha})\leq g((nW)^{\alpha})$. Therefore,
\begin{align*}
{\rm Tr}(V^{\alpha}(1-e^{-(nV)^{-\alpha}}))
&=n^{-\alpha}{\rm Tr}(g((nV)^{\alpha}))\leq n^{-\alpha}{\rm Tr}(g((nW)^{\alpha}))\\
&=\|V\|_{1,\infty}^{\alpha}\sum_{k=1}^{\infty}k^{-\alpha}(1-\exp(-(n\|V\|_{1,\infty})^{-\alpha}k^{\alpha}))\\
&\leq\|V\|_{1,\infty}^{\alpha}\int_0^{\infty}s^{-\alpha}(1-\exp(-(n\|V\|_{1,\infty})^{-\alpha}s^{\alpha}))ds\\
&=n^{1-\alpha}\|V\|_{1,\infty}\int_0^{\infty}u^{-\alpha}(1-\exp(-u^{\alpha}))du.
\end{align*}
Here, in the last step we used the substitution $s=n\|V\|_{1,\infty}u.$ This proves the first equality.

The second equality is proved as follows.
\begin{align*}
{\rm Tr}(e^{-(nV)^{-\alpha}}))&\leq {\rm Tr}(e^{-(nW)^{-\alpha}}))=\sum_{k=1}^{\infty}\exp(-(n\|V\|_{1,\infty})^{-\alpha}k^{\alpha})\\
&\leq\int_0^{\infty}\exp(-(n\|V\|_{1,\infty})^{-\alpha}s^{\alpha})ds=n\|V\|_{1,\infty}\int_0^{\infty}\exp(-u^{\alpha})du.
\end{align*}
In the last step we again used the substitution $s=n\|V\|_{1,\infty}u.$ 
\end{proof}


\begin{lem}\label{modulated lemma} 
If $0\leq V\in\mathcal{L}_{1,\infty}$ and if $A\in\mathcal{L}(H),$ then
$$
\sum_{k=0}^n\lambda(k,AV)={\rm Tr}(AVE_V[\frac1n,\infty))+O(1),\quad n\to\infty.
$$
\end{lem}
\begin{proof} 
Recall that a Hilbert-Schmidt operator $W$ is said to be $V-$modulated (in the sense of \cite[Definition 11.2.1]{LSZ}) if
$$\sup_{t>0}t^{1/2}\|W(1+tV)^{-1}\|_2<\infty.$$
We show that the operator $AV$ is $V-$modulated. Indeed, we have
$$
\sup_{t>0}t^{1/2}\|AV(1+tV)^{-1}\|_2
\leq\|A\|_{\infty}\sup_{t>0}t^{1/2}\|V(1+tV)^{-1}\|_2
\leq\|A\|_{\infty}\sup_{t>0}t^{1/2}\left\|\Big\{\frac1{\mu(k,V)^{-1}+t}\Big\}_{k\geq0}\right\|_2<\infty.
$$
Let $e_k$, $k\geq0$, be an eigenbasis of $V.$ Since $AV$ is $V-$modulated and since $V\geq0,$ it follows from Theorem 11.2.3 in \cite{LSZ} that
$$
\sum_{k=0}^n\lambda(k,AV)=\sum_{k=0}^n\mu(k,V)\langle Ae_k,e_k\rangle+O(1).
$$
By definition, $E_V[\frac1n,\infty)$ is the projection onto 
$e_k$, $0\leq k\leq m(n)$, where $m(n)={\rm Tr}(E_V[1/n,\infty))$. 
Since $V\in\mathcal{L}_{1,\infty}$, we have $\mu(k,V)\leq \frac {C}{k+1}$ for 
some constant $C>0$ and all $k\ge 0$. 
This inequality guarantees that $m(n)=O(n)$ as $n\to\infty$, 
by Equation \eqref{tlp}, in particular, there is a constant $C<\infty$ such that 
$m(n)\leq Cn$, for all $n\ge 1$. It may also happen that $m(n)<n$. 

If $m(n)<n,$ then (since $\mu(k,V)<\frac1n$ for $k>m(n)$), we have
$$\sum_{k=m(n)+1}^n\mu(k,V)\leq\sum_{k=m(n)+1}^n\frac1n\leq 1.$$
If $m(n)\geq n,$ then
$$\sum_{k=n+1}^{m(n)}\mu(k,V)\leq\sum_{k=n+1}^{Cn}\mu(k,V)\leq\|V\|_{1,\infty}\sum_{k=n+1}^{Cn}\frac1k=O(1).$$
In either case, we have
$$|\sum_{k=0}^n\mu(k,V)-\sum_{k=0}^{m(n)}\mu(k,V)|=O(1).$$

With these observations, we have the equality
$$
{\rm Tr}(AVE_V[\frac1n,\infty))=\sum_{k=0}^{m(n)}\mu(k,V)\langle Ae_k,e_k\rangle.
$$
It follows that
\begin{align*}
\left|\sum_{k=0}^n\lambda(k,AV)-{\rm Tr}(AVE_V[\frac1n,\infty))
\right|
&=\left|\sum_{k=0}^n\mu(k,V)\langle Ae_k,e_k\rangle+O(1)-\sum_{k=0}^{m(n)}\mu(k,V)\langle Ae_k,e_k\rangle\right|\\
&\leq\|A\|_{\infty}\left|\sum_{k=0}^n\mu(k,V)-\sum_{k=0}^{m(n)}\mu(k,V)\right|+O(1)=O(1).
\end{align*}
\end{proof}
The above Lemmas allow us to prove the first statement of Proposition \ref{measurability criterion}.

\smallskip

\begin{proof}[Proof of Proposition \ref{measurability criterion} \eqref{mesa}] 
We start by showing that
\begin{equation}\label{dds}
|{\rm Tr}(AVe^{-(nV)^{-\alpha}})-{\rm Tr}(AVE_V[1/n,\infty))
=O(1),\quad n\to\infty.
\end{equation}
\begin{align*}\mbox{Indeed,}\ \ 
&\left|{\rm Tr}\left(AVe^{-(nV)^{-\alpha}}\right)-{\rm Tr}(AVE_V[1/n,\infty))
\right|\\
&\qquad\leq\left|{\rm Tr}(AV(e^{-(nV)^{-\alpha}}-1)E_V[1/n,\infty))\right|
+\left|{\rm Tr}(AVe^{-(nV)^{-\alpha}}E_V[0,1/n))\right|\\
&\qquad\qquad\leq\|A\|_{\infty}\left(\,\left|{\rm Tr}(V(e^{-(nV)^{-\alpha}}-1)E_V[1/n,\infty))\right|+\left|{\rm Tr}(Ve^{-(nV)^{-\alpha}}E_V[0,1/n))\right|\,\right).
\end{align*}
In order to complete the proof, we observe that the spectral theorem yields 
$VE_V[0,1/n)\leq 1/n$. Similarly, for any $\alpha>1$ we have the inequality 
$\lambda\chi_{(1/n,\infty)}(\lambda)\leq n^{\alpha-1}\lambda^{\alpha}$, 
where $\chi_{(1/n,\infty)}$ is the indicator function of the interval ${(1/n,\infty)}$, and 
so the spectral theorem yields $VE_V[1/n,\infty)\leq n^{\alpha-1}V^{\alpha}$.

It now follows that
\begin{align*}
&
\left|{\rm Tr}\left(AVe^{-(nV)^{-\alpha}}\right)-{\rm Tr}(AVE_V[1/n,\infty))\right|
\\&\qquad\qquad
\leq\|A\|_{\infty}\left(n^{\alpha-1}{\rm Tr}(V^{\alpha}(1-e^{-(nV)^{-\alpha}}))
+\frac1n{\rm Tr}(e^{-(nV)^{-\alpha}})\right)=O(1).
\end{align*}
Here, the last equality holds by Lemma \ref{estimate}. Appealing to the assumption \eqref{meas eq crit} and Lemma \ref{modulated lemma}, we rewrite the preceding inequality as
$
\sum_{k=0}^n\lambda(k,AV)=z\log(n)+O(1)
$
and conclude using Proposition \ref{komm description}.
\end{proof}

To prove the second part of Proposition \ref{measurability criterion}, we need the following lemmas.

\begin{lem}\label{lt1} 
Let $\omega$ be a dilation invariant extended limit on $l_{\infty}.$ 
For every $0\leq V\in\mathcal{M}_{1,\infty}$ and $\alpha>1$, we have
$$
\omega\Big(\Big\{\frac1{n\log(n)}{\rm Tr}(e^{-(nV)^{-\alpha}})\Big\}_{n\geq0}\Big)=0.
$$
\end{lem}
\begin{proof} Fix $\varepsilon\in[0,1]$ and observe that
$
e^{-t^{-\alpha}}\leq4\varepsilon t^2,\quad 0\leq t\leq\varepsilon.
$
Hence, for every $t>0,$ we have
$$
e^{-(nt)^{-\alpha}}\leq\chi_{(\frac{\varepsilon}n,\infty)}(t)+4\varepsilon(nt)^2\chi_{[0,\frac{\varepsilon}{n}]}(t)\leq \chi_{(\frac{\varepsilon}n,\infty)}(t)+4\varepsilon(\min\{nt,1\})^2.
$$
Applying the functional calculus, we infer from the inequality above that
$$
e^{-(nV)^{-\alpha}}\leq E_V({\varepsilon}/n,\infty)+4\varepsilon\min\{(nV),1\}^2.
$$
Hence, using the fact that $\omega$ is a positive functional, we obtain 
$$
\omega\left(\Big\{\frac1{n\log(n)}{\rm Tr}(e^{-(nV)^{-\alpha}})\Big\}_{n\geq0}\right)
\leq\omega\left(\Big\{\frac{n_V(\frac{\varepsilon}{n})}{n\log(n)}\Big\}_{n\geq0}\right)
+4\varepsilon\omega\left(\Big\{\frac{\min\{(nV),1\}^2}{n\log(n)}\Big\}_{n\geq0}\right).
$$
Here, the second term is well defined thanks to Lemma 8.4.2 (b) in \cite{LSZ}. By Lemma 8.2.8 in \cite{LSZ}, the first term vanishes for every $\varepsilon>0.$ Letting $\varepsilon\to0,$ we conclude the proof.
\end{proof}

\begin{lem}\label{lt2} Let $\omega$ be a dilation invariant extended limit on $l_{\infty}$, $\alpha>1$, 
and introduce the notation $T_+$ for the positive part of a self adjoint operator $T$. For every $A\in\mathcal{L}(H)$ and for every $0\leq V\in\mathcal{M}_{1,\infty},$ we have
$$
\omega\Big(\Big\{\frac1{\log(n)}{\rm Tr}(AVe^{-(nV)^{-\alpha}})\Big\}_{n\geq0}\Big)
=\omega\Big(\Big\{\frac1{\log(n)}{\rm Tr}(A(V-1/n)_+)\Big\}_{n\geq0}\Big).
$$
\end{lem}
\begin{proof} Without loss of generality, the operator $A$ is positive. 
Fix $\varepsilon>0.$ Applying the functional calculus to the numerical inequality
$$
e^{-\varepsilon^{\alpha}}(t-1/{\varepsilon n})_+\leq te^{-(nt)^{-\alpha}}
\leq(t-1/n)_++\frac1n\chi_{[\frac1n,\infty)}(t)+\frac1ne^{-(nt)^{\alpha}},
$$
(the subscripted $+$ again denotes the positive part)
we obtain an inequality involving trace class operators
\begin{equation}\label{lt2eq3}
e^{-\varepsilon^{\alpha}}(V-1/{\varepsilon n})_+\leq Ve^{-(nV)^{-\alpha}}
\leq(V-1/n)_++\frac1nE_V[1/n,\infty)+\frac1ne^{-(nV)^{\alpha}}.
\end{equation}
For any trace class operator $T$, cyclicity of the trace gives 
${\rm Tr}(A^{1/2} T A^{1/2})= {\rm Tr} (AT)$. We apply this observation 
to the second inequality in \eqref{lt2eq3}  to infer that 
$$
{\rm Tr}(AVe^{-(nV)^{-\alpha}})\leq{\rm Tr}\left(A(V-1/n)_+\right)
+\frac{\|A\|_{\infty}}{n}n_V(1/n)+\frac{\|A\|_{\infty}}{n}{\rm Tr}(e^{-(nV)^{\alpha}}).
$$
It follows from Lemma 8.2.8 in \cite{LSZ} and Lemma \ref{lt1} that
\begin{equation}\label{lt2eq1}
\omega\Big(\Big\{\frac1{\log(n)}{\rm Tr}(AVe^{-(nV)^{-\alpha}})\Big\}_{n\geq0}\Big)
\leq\omega\Big(\Big\{\frac1{\log(n)}{\rm Tr}\left(A(V-1/n)_+\right)\Big\}_{n\geq0}\Big).
\end{equation}
Now we apply ${\rm Tr}(A^{1/2} T A^{1/2})= {\rm Tr} (AT)$  to the first  inequality in \eqref{lt2eq3} 
to insert a positive operator $A$ under the trace. So we infer that
$$
\omega\left(\Big\{\frac1{\log(n)}{\rm Tr}(AVe^{-(nV)^{-\alpha}})\Big\}_{n\geq0}\right)
\geq e^{-\varepsilon^{\alpha}}\omega\left(\Big\{\frac1{\log(n)}{\rm Tr}\left(A(V-1/n\varepsilon)_+\right)\Big\}_{n\geq0}\right).
$$
Taking into account that $\omega$ is 
dilation invariant and passing to the limit $\varepsilon\to0$, we infer that
\begin{equation}\label{lt2eq2}
\omega\Big(\Big\{\frac1{\log(n)}{\rm Tr}(AVe^{-(nV)^{-\alpha}})\Big\}_{n\geq0}\Big)
\geq\omega\Big(\Big\{\frac1{\log(n)}{\rm Tr}\left(A(V-1/n)_+\right)\Big\}_{n\geq0}\Big).
\end{equation}
The assertion follows by combining \eqref{lt2eq1} and \eqref{lt2eq2}.
\end{proof}

We can now complete the proof of Proposition \ref{measurability criterion}.

\begin{proof}[Proof of Proposition \ref{measurability criterion} \eqref{mesb}] 
For every dilation invariant extended limit $\omega$ on $l_{\infty},$ we define a heat semigroup functional
$$
\xi_{\omega}:W\to(\omega\circ M)\Big(\Big\{\frac1n{\rm Tr}(e^{-(nW)^{-1}})\Big\}_{n\geq0}\Big),
\quad 0\leq W\in\mathcal{M}_{1,\infty}.
$$
By Theorem 8.2.5 in \cite{LSZ}, the functional $\xi_{\omega}$ 
extends to a Dixmier trace on $\mathcal{M}_{1,\infty}.$ For every 
dilation invariant extended limit $\omega,$ we infer from Lemma \ref{lt2} that
$$
\omega\Big(\Big\{\frac1{\log(n)}{\rm Tr}\left(A(V-\frac1n)_+\right)\Big\}_{n\geq0}\Big)
=\omega\Big(\Big\{\frac1{\log(n)}{\rm Tr}(AVe^{-(nV)^{-\alpha}})\Big\}_{n\geq0}\Big)=z.
$$
Then, by Lemma 8.5.3 in \cite{LSZ}, we have
$
\xi_{\omega}(AV)=z
$
for every dilation invariant extended limit $\omega.$ Finally, by 
Theorem 8.3.6 in \cite{LSZ}, the set of all Dixmier traces 
coincides with the set of all functionals $\xi_{\omega},$ where 
$\omega$ runs through all dilation invariant extended limits on $l_{\infty}.$ 
The assertion follows immediately.
\end{proof}

\section{Preliminaries on noncommutative geometry and the statements of the main results}
\label{sec:prelim-NCG}
\subsection{Spectral triples and Hochschild (co)homology}

 Let $D:{\rm dom}(D)\to H$ be a self-adjoint operator with 
 ${\rm dom}(D)\subset H$ a dense linear subspace.
An operator $D$ admits a polar decomposition $D=F|D|,$ where the phase $F$ is a  
self-adjoint unitary operator defined by 
$F:=E_D([0,\infty))-E_D(-\infty, 0)$ and $|D|:{\rm dom}(D)\to H$ is a self-adjoint operator. 
The following definitions should be compared with Definition 1.20 in \cite{CGRS2}.

\begin{defi}
A spectral triple $(\mathcal{A},H,D)$ consists of a subalgebra $\mathcal{A}$  of  $\mathcal{L}(H)$ such that:
\begin{enumerate}[{\rm (a)}]
\item $a:{\rm dom}(D)\to{\rm dom}(D)$ for all $a\in\mathcal{A}$;
\item $[D,a]:{\rm dom}(D)\to H$ extends to an operator $\partial(a)\in\mathcal{L}(H)$ for all $a\in\mathcal{A}$;
\item $a(1+D^2)^{-1/2}$ is a compact operator for all $a\in \mathcal{A}$.
\end{enumerate}
\end{defi}

In what follows, if $a:{\rm dom}(D)\to{\rm dom}(D),$ then the (a priori unbounded) operator $[|D|,a]:{\rm dom}(D)\to H$ is denoted by $\delta(a).$

\begin{defi}\label{qc}
A spectral triple is $QC^{\infty}$ if
\begin{enumerate}[{\rm (a)}]
\item $a:{\rm dom}(D^n)\to{\rm dom}(D^n)$ for all $a\in\mathcal{A}$ and 
\item for all $n\geq0$ the operators
$\ \ \delta^n(a):{\rm dom}(D^n)\to H,\quad\delta^n(\partial(a)):{\rm dom}(D^{n+1})\to H\ $
extend to bounded operators for all $n\geq0$ and for all $a\in\mathcal{A}.$
\end{enumerate}
\end{defi}

\begin{defi}\label{oddeven} A spectral triple is said to be
\begin{enumerate}[{\rm (a)}]
\item even if there exists $\Gamma\in\mathcal{L}(H)$ such that $\Gamma=\Gamma^*,$ $\Gamma^2=1$ and such that $[\Gamma,a]=0$ for all $a\in\mathcal{A},$ $\{D,\Gamma\}=0.$ Here $\{\cdot,\cdot\}$ denotes anticommutator.
\item odd if no such $\Gamma$ exists. In this case, we set $\Gamma=1.$
\item $(p,\infty)-$summable if $(1+D^2)^{-p/2}\in\mathcal{L}_{1,\infty}.$
\item $\mathcal{M}_{1,\infty}^{(p)}-$summable if $(1+D^2)^{-p/2}\in\mathcal{M}_{1,\infty}.$
\end{enumerate}
\end{defi}

The following assertion is proved in many places, for example \cite[Corollary 0.5]{CPS1}, \cite{CPRS}, 
and \cite{PS}. We prove a related statement in Lemma \ref{f second est}.

\begin{prop}\label{f der def} If $(\mathcal{A},H,D)$ is a spectral triple that is 
$QC^{\infty}$ and $(p,\infty)-$summable, then $[F,a]$ and $ [F,\delta^k(a)]$ lie in 
$\mathcal{L}_{p,\infty}$ for all $a\in\mathcal{A}$ and $k\ge 1$.
\end{prop}

Define  multilinear mappings  $ch:\mathcal{A}^{\otimes (p+1)}\to\mathcal{L}(H)$ and 
$\Omega:\mathcal{A}^{\otimes (p+1)}\to\mathcal{A}$ by setting
$$
ch(a_0\otimes\cdots\otimes a_p)=F\Gamma\prod_{k=0}^p[F,a_k],
\quad\quad
\Omega(a_0\otimes\cdots\otimes a_p)=\Gamma a_0\prod_{k=1}^p[D,a_k].
$$

If a spectral triple $(\mathcal{A},H,D)$ is $(p,\infty)-$summable, 
then it follows from Proposition \ref{f der def} and the H\"older inequality in Equation
\eqref{lpi mult} that $ch(c)\in\mathcal{L}_{p/(p+1),\infty}\subset\mathcal{L}_1$ 
for all $c\in\mathcal{A}^{\otimes (p+1)}.$ This justifies the following definition.

\begin{defi} If $(\mathcal{A},H,D)$ is a $(p,\infty)-$summable spectral triple, 
then Connes'  Chern character $\mathcal{A}^{\otimes (p+1)}\to\mathbb{C}$ is defined by setting \quad
$$
Ch(c)=\frac{1}{2}{\rm Tr}(ch(c)),\quad c\in\mathcal{A}^{\otimes (p+1)}.
$$
\end{defi}
In fact the Chern character is the class of $Ch$ in periodic cyclic cohomology, but we
shall ignore this distinction in the sequel. 

We now turn to Hochschild (co)homology.
The algebra $\mathcal A$ is equipped with the $\delta$-topology, \cite{RenSmo}, 
determined by the seminorms $q_n:\A\to[0,\infty)$ given by
$$
q_n(a)=\sum_{k=0}^n\Vert \delta^k(a)\Vert+\Vert\delta^k([D,a])\Vert.
$$

The tensor powers of 
$\mathcal A$ are completed in the projective tensor product topology.
If $\theta:\mathcal{A}^{\otimes n}\to\mathbb{C}$ is a continuous 
multilinear functional, then the multilinear functional 
$b\theta:\mathcal{A}^{\otimes (n+1)}\to\mathbb{C}$ is defined by 
\begin{align*}
(b\theta)(a_0\otimes\cdots\otimes a_n)&=\theta(a_0a_1\otimes a_2\otimes\cdots\otimes a_n)\\
&+\sum_{k=1}^{n-1}(-1)^k\theta(a_0\otimes a_1\otimes\cdots a_{k-1}\otimes a_{k}a_{k+1}\otimes a_{k+1}\otimes\cdots\otimes a_n)\\
&+(-1)^n\theta(a_na_0\otimes a_1\otimes a_2\otimes\cdots\otimes a_{n-1}).
\end{align*}
We call $b\theta$ the Hochschild coboundary of $\theta.$ If $b\theta=0,$ then we call $\theta$ a Hochschild cocycle.
We also need the dual notion of Hochschild cycle. The Hochschild boundary $b:\mathcal{A}^{\otimes (n+1)}\to\mathcal{A}^{\otimes n}$ is defined by setting
\begin{align*}
b(a_0\otimes\cdots\otimes a_n)&=a_0a_1\otimes a_2\otimes\cdots\otimes a_n\\
&+\sum_{k=1}^{n-1}(-1)^ka_0\otimes a_1\otimes\cdots a_{k-1}\otimes a_{k}a_{k+1}\otimes a_{k+1}\otimes\cdots\otimes a_n\\
&+(-1)^n a_na_0\otimes a_1\otimes a_2\otimes\cdots\otimes a_{n-1}.
\end{align*}
If $c\in\mathcal{A}^{\otimes (n+1)}$ is such that $bc=0,$ then $c$ is called a Hochschild cycle.
For example, if $n=1,$ then $b(a_0\otimes a_1)=[a_0,a_1].$ Hence, an elementary tensor $a_0\otimes a_1$ is a Hochschild cycle if and only if $a_0$ and $a_1$ commute.
The definitions are dual in the sense that for any multilinear functional $\theta$,
$(b\theta)(a)=\theta(ba)$.
In particular, a Hochschild coboundary vanishes on every Hochschild cycle.
\subsection{The main results and the plan of the proofs}\label{plan}

For the statement of our main theorem, and the remainder of the paper,
we assume that $p\in\mathbb{N}.$

\begin{thm}\label{main result} 
Let $(\mathcal{A},H,D)$ be a $QC^{\infty}$ spectral triple which is even or odd according to
whether $p$ is even or odd, and 
let $c\in\mathcal{A}^{\otimes (p+1)}$ be a Hochschild cycle.
\begin{enumerate}[{\rm (a)}]
\item\label{maina} If the spectral triple is $(p,\infty)-$summable, then
for every normalised trace $\varphi$ on $\mathcal{L}_{1,\infty}$
\begin{equation}
\label{main equation}
\varphi(\Omega(c)(1+D^2)^{-p/2})=Ch(c).
\end{equation}
\item\label{mainb} If the spectral triple is $\mathcal{M}_{1,\infty}^{(p)}-$summable, then
${\rm Tr}_{\omega}(\Omega(c)(1+D^2)^{-p/2})=Ch(c)$
for every Dixmier trace on $\mathcal{M}_{1,\infty}.$
\end{enumerate}
\end{thm}
Let us illustrate the assertion for $p=1.$ If elements $a_0,a_1\in\mathcal{A}$ commute, then the elementary tensor $a_0\otimes a_1$ is a Hochschild $1$-cycle and
$$
\varphi(a_0[D,a_1](1+D^2)^{-1/2})=\frac12{\rm Tr}(F[F,a_0][F,a_1])
$$
for every trace $\varphi$ on $\mathcal{L}_{1,\infty}.$
The corollary below follows from Theorem \ref{main result} and Proposition \ref{komm description}.

\begin{cor}\label{main cor} 
Suppose that the assumptions of Theorem \ref{main result} \eqref{maina} hold. 
If $c\in\mathcal{A}^{\otimes (p+1)}$ is a Hochschild cycle, then:
\begin{enumerate}[{\rm (a)}]
\item $\Omega(c)(1+D^2)^{-p/2}\in[\mathcal{L}_{1,\infty},\mathcal{L}(H)]$ 
if and only if $Ch(c)=0$, and more generally
$$
\Omega(c)(1+D^2)^{-p/2}\in Ch(c)\cdot{\rm diag}\Big(\Big\{\frac1{k+1}\Big\}_{k\geq0}\Big)
+[\mathcal{L}_{1,\infty},\mathcal{L}(H)];
$$
\item there is an equality\quad
$
\sum_{m=0}^n\lambda(m,\Omega(c)(1+D^2)^{-p/2})=Ch(c)\log(n)+O(1),\quad n\geq0.
$
\end{enumerate}
\end{cor}

%

Theorem \ref{main result} is initially proved under the assumption of
invertibility of $D$ in subsection \ref{prmain}, after proving some intermediate
steps. The first step is to replace $\Omega(c)|D|^{-p}$ by a new operator. More specifically,
for $1\leq m\leq p,$ we define the multilinear mappings 
$\mathcal{W}_m:\mathcal{A}^{\otimes(p+1)}\to\mathcal{L}(H)$ by setting
\begin{equation}
\mathcal{W}_m(a_0\otimes\cdots\otimes a_p)=\Gamma a_0\Big(\prod_{k=1}^{m-1}[F,a_k]\Big)\delta(a_m)\Big(\prod_{k=m+1}^p[F,a_k]\Big).
\label{eq:W-emm}
\end{equation}
By Proposition \ref{f der def} and by the H\"older property in Equation \eqref{lpi mult}, we have 
$\mathcal{W}_m(c)D^{-1}\in\mathcal{L}_{1,\infty}$ 
(respectively, $\mathcal{W}_m(c)D^{-1}\in\mathcal{M}_{1,\infty}$).
Then, by exploiting Hochschild cohomology (see Appendix A),  we show
in subsection \ref{ssec} that (for $D^{-1}\in\mathcal{L}_{1,\infty}$)
$$
\Omega(c)|D|^{-p}-p\mathcal{W}_p(c)D^{-1}\in[\mathcal{L}_{1,\infty},\mathcal{L}(H)].
$$
We prove the analogous result for $D^{-1}\in\mathcal{M}_{1,\infty}$ also.
Then, in subsection \ref{commutator section}, 
we obtain a number of commutator estimates which allow us to
prove, in subsection \ref{heat}, that
for every Hochschild cycle $c\in\mathcal{A}^{\otimes (p+1)}$,
$$
{\rm Tr}(\mathcal{W}_p(c)D^{-1}e^{-(s|D|)^{p+1}})
=Ch(c)\log\big(1/s\big)+O(1),\quad s\to0.
$$
By invoking our abstract measurability criterion, Proposition \ref{measurability criterion}, we can
then assemble the pieces to prove the main result in subsection \ref{prmain}. We also
show at this point how to remove the invertibility assumption.

\section{Proofs}
\label{sec:proofs}

Until subsection \ref{prmain}, we will suppose that the operator $D$ of a 
spectral triple $(\mathcal{A},H,D)$ is invertible.

\subsection{Exploiting  Hochschild cohomology}\label{ssec}
Our aim in this subsection is to prove the following result, by refining the 
approach of \cite[Section 3.5]{CPRS}.

\begin{prop}\label{reduction} Let $(\mathcal{A},H,D)$ be an  odd (respectively, even) 
$QC^{\infty}$ spectral triple and let $p$ be odd (respectively, even). For every Hochschild cycle $c\in\mathcal{A}^{\otimes (p+1)}$ we have:
\begin{enumerate}[{\rm (a)}]
\item\label{reda} if $D^{-p}\in\mathcal{L}_{1,\infty},$ then
$
\Omega(c)|D|^{-p}-p\mathcal{W}_p(c)D^{-1}\in[\mathcal{L}_{1,\infty},\mathcal{L}(H)],
$
\item\label{redb} if $D^{-p}\in\mathcal{M}_{1,\infty},$ then
$
\Omega(c)|D|^{-p}-p\mathcal{W}_p(c)D^{-1}\in[\mathcal{M}_{1,\infty},\mathcal{L}(H)].
$
\end{enumerate}
\end{prop}

We consider auxiliary multilinear mappings which generalise 
the mappings $\mathcal{W}_m,$ $1\leq m\leq p$, introduced above in Equation \eqref{eq:W-emm}. 
For $\mathscr{A}\subset \{1,\dots,p\}$ define the multilinear mapping 
$\mathcal{W}_{\mathscr{A}}:\mathcal{A}^{\otimes (p+1)}\to\mathcal{L}(H)$ by setting
$$
\mathcal{W}_{\mathscr{A}}(a_0\otimes\cdots\otimes a_p)
:=\Gamma a_0\prod_{k=1}^p[b_k,a_k],
\quad a_0\otimes\cdots\otimes a_p\in\mathcal{A}^{\otimes (p+1)},
$$
where $b_k=|D|,$ for $k\in\mathscr{A},$ and $b_k=F,$ for $k\notin\mathscr{A}.$ 
Evidently, if $\mathscr{A}=\{m\},$ then $\mathcal{W}_{\mathscr{A}}=\mathcal{W}_m.$ 
It follows from Proposition \ref{f der def} and the H\"older property in Equation \eqref{lpi mult} that
$$
\mathcal{W}_{\mathscr{A}}(a)D^{-|\mathscr{A}|}\in\mathcal{L}_{1,\infty},
\quad \mathscr{A}\subset \{1,\dots,p\}.
$$

For every $\mathscr{A}\subset \{1,\dots,p\},$ define the number
$$
n_{\mathscr{A}}=|\{(i,j):\ i<j,\ i\in\mathscr{A},\ j\notin\mathscr{A}\}|.
$$
The following assertion explains the introduction of the mappings 
$\mathcal{W}_{\mathscr{A}},$ $\mathscr{A}\subset \{1,\dots,p\}$ 
that are used for the proof of Proposition \ref{reduction}.  
We denote the cardinality of $\mathscr{A}$ by $|\mathscr{A}|$.
\begin{lem}\label{kogom6} If $(\mathcal{A},H,D)$ is $QC^{\infty}$ 
spectral triple with $D^{-p}\in\mathcal{L}_{1,\infty},$ then for all $c\in\mathcal{A}^{\otimes (p+1)}$
$$
\Omega(c)|D|^{-p}-\sum_{\mathscr{A}\subset \{1,\dots,p\}}(-1)^{n_{\mathscr{A}}}\mathcal{W}_{\mathscr{A}}(c)D^{-|\mathscr{A}|}\in\mathcal{L}_1.
$$
\end{lem}
\begin{proof} 
We will proceed by proving  that for $1\leq q\leq p$
and $c=a_0\otimes a_1\otimes\cdots\otimes a_q$,
\begin{equation}
\Gamma a_0[D,a_1]\cdots[D,a_q]|D|^{-q}
=\sum_{\mathscr{A}\subset \{1,\dots,q\}}(-1)^{n_{\mathscr{A}}}
\mathcal{W}_{\mathscr{A}}(c)D^{-|\mathscr{A}|}\bmod\mathcal{L}_{p/(q+1),\infty}.
\label{eq:q<p}
\end{equation}

For $q=1,$ we consider $c=a_0\otimes a_1\in\mathcal{A}^{\otimes 2}.$ We have
\begin{align*}
[D,a_1]=[F|D|,a_1]&=F\delta(a_1)+[F,a_1]|D|=[F,\delta(a_1)]+\delta(a_1)F+[F,a_1]|D|\\
&=([F,\delta(a_1)]|D|^{-1}+\delta(a_1)D^{-1}+[F,a_1])|D|.
\end{align*}
By Proposition \ref{f der def} and the assumption, the operator $[F,\delta(a_1)]|D|^{-1}$ is in $\mathcal{L}_{p,\infty}\cdot \mathcal{L}_{p,\infty}
\subset \mathcal{L}_{p/2,\infty}$, while the other terms in parentheses are in 
$\mathcal{L}_{p,\infty}$, and give the right hand side of Equation \eqref{eq:q<p}. Thus we have
proved the case $q=1$.

Suppose then that we have proved the claim for some $q<p$. Since commutators with $|D|^{-1}$ improve summability, it follows that
$$\Big(\prod_{k=2}^{q+1}[D,a_k]\Big)|D|^{-1}=|D|^{-1}\Big(\prod_{k=2}^{q+1}[D,a_k]\Big)\bmod \mathcal{L}_{p/2,\infty}.$$
Therefore,
$$\Gamma a_0\Big(\prod_{k=1}^{q+1}[D,a_k]\Big)|D|^{-q-1}=\Gamma a_0[D,a_1]\Big(\Big(\prod_{k=2}^{q+1}[D,a_k]\Big)|D|^{-1}\Big)|D|^{-q}=$$
$$=\Gamma a_0[D,a_1]\Big(|D|^{-1}\Big(\prod_{k=2}^{q+1}[D,a_k]\bmod\mathcal{L}_{p/2,\infty}\Big)\Big)|D|^{-q}=$$
$$=\Gamma a_0[D,a_1]|D|^{-1}\Big(\Big(\prod_{k=2}^{q+1}[D,a_k]\Big)|D|^{-q}\Big)\bmod\mathcal{L}_{p/(q+2),\infty}.$$
By induction, we have
$$\Big(\prod_{k=2}^{q+1}[D,a_k]\Big)|D|^{-q}=\sum_{\mathscr{A}\subset\{2,\dots,q+1\}}\Gamma \mathcal{W}_{\mathscr{A}}(1,a_2,\dots,a_{q+1})
(-1)^{n_{\mathscr{A}}}D^{-|\mathscr{A}|}\bmod\mathcal{L}_{p/(q+1),\infty}.$$
Thus,
$$\Gamma a_0\Big(\prod_{k=1}^{q+1}[D,a_k]\Big)|D|^{-q-1}=$$
$$=\Gamma a_0[D,a_1]|D|^{-1}\Big(\sum_{\mathscr{A}\subset\{2,\dots,q+1\}}\Gamma \mathcal{W}_{\mathscr{A}}(1,a_2,\dots,a_{q+1})
(-1)^{n_{\mathscr{A}}}D^{-|\mathscr{A}|}\Big)\bmod\mathcal{L}_{p/(q+2),\infty}=$$
$$=\Gamma a_0\Big(\delta(a_1)F+[F,a_1]|D|\Big)|D|^{-1}\Big(\sum_{\mathscr{A}\subset\{2,\dots,q+1\}}\Gamma \mathcal{W}_{\mathscr{A}}(1,a_2,\dots,a_{q+1})
(-1)^{n_{\mathscr{A}}}D^{-|\mathscr{A}|}\Big)\bmod\mathcal{L}_{p/(q+2),\infty}$$
Since commutators with $|D|^{-1}$ improve summability, it follows that
$$|D|^{-1}\Gamma \mathcal{W}_{\mathscr{A}}(1,a_2,\dots,a_{q+1})=\Gamma \mathcal{W}_{\mathscr{A}}(1,a_2,\dots,a_{q+1})|D|^{-1}\bmod\mathcal{L}_{p/(q+2-|\mathscr{A}|),\infty}.$$
Since $[F,\delta(a)]\in\mathcal{L}_{p,\infty}$ for all $a\in\mathcal{A},$ it follows that
$$F\Gamma \mathcal{W}_{\mathscr{A}}(1,a_2,\dots,a_{q+1})=(-1)^{q-|\mathscr{A}|}\Gamma \mathcal{W}_{\mathscr{A}}(1,a_2,\dots,a_{q+1})F\bmod\mathcal{L}_{p/(q+1-|\mathscr{A}|),\infty}.$$
Indeed, we have $F[F,a]=-[F,a]F$ for every $a\in\mathcal{A}$ and there are exactly $q-|\mathscr{A}|$ commutators $[F,a_j]$ in $\mathcal{W}_{\mathscr{A}}.$

Therefore,
$$\Gamma a_0\Big(\prod_{k=1}^{q+1}[D,a_k]\Big)|D|^{-q-1}=$$
$$=\sum_{\mathscr{A}\subset\{2,\dots,q+1\}}(-1)^{n_{\mathscr{A}}}(-1)^{q-|\mathscr{A}|}\Gamma a_0\delta(a_1)\Gamma \mathcal{W}_{\mathscr{A}}(1,a_2,\dots,a_{q+1})D^{-|\mathscr{A}|-1}+$$
$$+\sum_{\mathscr{A}\subset\{2,\dots,q+1\}}(-1)^{n_{\mathscr{A}}}\Gamma a_0[F,a_1]\Gamma \mathcal{W}_{\mathscr{A}}(1,a_2,\dots,a_{q+1})D^{-|\mathscr{A}|}\bmod\mathcal{L}_{p/(q+2),\infty}.$$

For each $\mathscr{A}\subset\{2,\dots,q+1\}$ define $\tilde{\mathscr{A}}=\mathscr{A}\cup\{1\}\subset\{1,\dots,q+1\}$
and $\hat{\mathscr{A}}=\mathscr{A}\subset\{1,\dots,q+1\}.$ Then $n_{\tilde{\mathscr{A}}}=q-|\mathscr{A}|+n_{\mathscr{A}}$ while 
$n_{\hat{\mathscr{A}}}=n_{\mathscr{A}}.$

By definition, we have
$$\Gamma a_0\delta(a_1)\Gamma\mathcal{W}_{\mathscr{A}}(1,a_2,\dots,a_{q+1})=\mathcal{W}_{\tilde{\mathscr{A}}}(c)$$
and
$$\Gamma a_0[F,a_1]\Gamma \mathcal{W}_{\mathscr{A}}(1,a_2,\dots,a_{q+1})=\mathcal{W}_{\hat{\mathscr{A}}}(c).$$
Hence,
$$\Gamma a_0\Big(\prod_{k=1}^{q+1}[D,a_k]\Big)|D|^{-q-1}=\sum_{\mathscr{A}\subset\{2,\dots,q+1\}}(-1)^{n_{\tilde{\mathscr{A}}}}\mathcal{W}_{\tilde{\mathscr{A}}}(c)+\sum_{\mathscr{A}\subset\{2,\dots,q+1\}}(-1)^{n_{\hat{\mathscr{A}}}}\mathcal{W}_{\hat{\mathscr{A}}}(c)\bmod\mathcal{L}_{p/(q+2),\infty}.$$
Since every $\mathscr{B}\subset\{1,\dots,q+1\}$ coincides either with $\tilde{\mathscr{A}}$ or else with $\hat{\mathscr{A}}$ for a unique $\mathscr{A}\subset\{2,\dots,q+1\},$ the equation \eqref{eq:q<p} follows for $q+1.$ This proves the Lemma.
\end{proof}

\begin{lem}\label{kogom2} Let $(\mathcal{A},H,D)$ be a 
$QC^{\infty}$ spectral triple and let 
$c\in\mathcal{A}^{\otimes (p+1)}$ be a Hochschild cycle. 
Suppose that $|\mathscr{A}|\geq 2$ and $m-1,m\in\mathscr{A}$ for some $m.$
\begin{enumerate}[{\rm (a)}]
\item If $D^{-p}\in\mathcal{L}_{1,\infty},$ then
$
\mathcal{W}_{\mathscr{A}}(c)D^{-|\mathscr{A}|}\in[\mathcal{L}_{1,\infty},\mathcal{L}(H)].
$
\item If $D^{-p}\in\mathcal{M}_{1,\infty},$ then
$
\mathcal{W}_{\mathscr{A}}(c)D^{-|\mathscr{A}|}\in[\mathcal{M}_{1,\infty},\mathcal{L}(H)].
$
\end{enumerate}
\end{lem}
\begin{proof} Let $\varphi$ be a trace on $\mathcal{L}_{1,\infty}$ (respectively, on $\mathcal{M}_{1,\infty}$). The mapping on $\mathcal{A}^{\otimes (p+1)}$
given by 
$$c\to\varphi(\mathcal{W}_{\mathscr{A}}(c)D^{-|\mathscr{A}|})$$
is the Hochschild coboundary (see Appendix \ref{cobos}) of the multilinear mapping
defined by
$$
a_0\otimes\cdots\otimes a_{p-1}\mapsto
\frac{(-1)^{m-1}}{2}\varphi\left(\Gamma a_0\prod_{k=1}^{m-2}[b_k,a_k]\delta^2(a_{m-1})
\prod_{k=m}^{p-1}[b_{k+1},a_k]D^{-|\mathscr{A}|}\right).
$$
Since a Hochschild coboundary vanishes on every Hochschild cycle, it follows that
$
\varphi(\mathcal{W}_{\mathscr{A}}(c)D^{-|\mathscr{A}|})=0
$
for every Hochschild cycle $c\in\mathcal{A}^{\otimes (p+1)}.$ 
Since $\varphi$ is an arbitrary trace, the assertion follows.
\end{proof}

\begin{lem}\label{kogom3} Let $(\mathcal{A},H,D)$ be a $QC^{\infty}$ spectral triple and let $c\in\mathcal{A}^{\otimes (p+1)}$ be a Hochschild cycle. Suppose that $|\mathscr{A}_1|=|\mathscr{A}_2|\geq 2$ and that the symmetric difference $\mathscr{A}_1\Delta\mathscr{A}_2=\{m-1,m\}$ for some $m.$
\begin{enumerate}[{\rm (a)}]
\item If $D^{-p}\in\mathcal{L}_{1,\infty},$ then
$\mathcal{W}_{\mathscr{A}_1}(c)D^{-|\mathscr{A}_1|}+\mathcal{W}_{\mathscr{A}_2}(c)D^{-|\mathscr{A}_2|}\in[\mathcal{L}_{1,\infty},\mathcal{L}(H)].$
\item If $D^{-p}\in\mathcal{M}_{1,\infty},$ then
$\mathcal{W}_{\mathscr{A}_1}(c)D^{-|\mathscr{A}_1|}+\mathcal{W}_{\mathscr{A}_2}(a)D^{-|\mathscr{A}_2|}\in[\mathcal{M}_{1,\infty},\mathcal{L}(H)].$
\end{enumerate}
\end{lem}
\begin{proof} Let $\varphi$ be a trace on $\mathcal{L}_{1,\infty}$ (respectively, on $\mathcal{M}_{1,\infty}$). The mapping on $\mathcal{A}^{\otimes (p+1)}$ given by
$$c\to\varphi(\mathcal{W}_{\mathscr{A}_1}(c)D^{-|\mathscr{A}_1|})+\varphi(\mathcal{W}_{\mathscr{A}_2}(c)D^{-|\mathscr{A}_2|})$$
is the Hochschild coboundary (see Appendix \ref{cobos}) of the multilinear mapping
defined by
$$
a_0\otimes\cdots\otimes a_{p-1}\to(-1)^{m-1}\varphi\left(\Gamma a_0\prod_{k=1}^{m-2}
[b_k,a_k][F,\delta(a_{m-1})]\prod_{k=m}^{p-1}[b_{k+1},a_k]D^{-|\mathscr{A}_1|}\right).
$$
The proof is concluded by using the same argument as in the preceding lemma.
\end{proof}

\begin{cor}\label{kogom4} Let $(\mathcal{A},H,D)$ be a $QC^{\infty}$ spectral triple and let $c\in\mathcal{A}^{\otimes (p+1)}$ be a Hochschild cycle. Suppose that $|\mathscr{A}|\geq 2.$
\begin{enumerate}[{\rm (a)}]
\item If $D^{-p}\in\mathcal{L}_{1,\infty},$ then
$\mathcal{W}_{\mathscr{A}}(c)D^{-|\mathscr{A}|}\in[\mathcal{L}_{1,\infty},\mathcal{L}(H)].$
\item If $D^{-p}\in\mathcal{M}_{1,\infty},$ then
$\mathcal{W}_{\mathscr{A}}(c)D^{-|\mathscr{A}|}\in[\mathcal{M}_{1,\infty},\mathcal{L}(H)].$
\end{enumerate}
\end{cor}
\begin{proof} Let $n<m$ be such that $n,m\in\mathscr{A}.$ Without loss of generality, $i+n\notin\mathscr{A}$ for all $0<i<m-n.$ Set
$$\mathscr{A}_i=(\mathscr{A}\backslash\{n\})\cup\{i+n\},\quad 0\leq i<m-n.$$
We have
\begin{enumerate}
\item $|\mathscr{A}_i|=|\mathscr{A}|$ and $|\mathscr{A}_i\Delta\mathscr{A}_{i-1}|=2$ for all $1\leq i<m-n.$
\item $\mathscr{A}_0=\mathscr{A}$ and $m-1,m\in\mathscr{A}_{m-n-1}.$
\end{enumerate}
It follows from Lemma \ref{kogom3} that $\mathcal{W}_{\mathscr{A}_{m-n-1}}(a)D^{-1}\in[\mathcal{L}_{1,\infty},\mathcal{L}(H)]$ (respectively, $\mathcal{W}_{\mathscr{A}_{m-n-1}}(a)D^{-1}\in[\mathcal{M}_{1,\infty},\mathcal{L}(H)]$). The assertion follows by applying Lemma \ref{kogom2} $m-n-1$ times.
\end{proof}

\begin{lem}\label{kogom5}  Let $(\mathcal{A},H,D)$ be a $QC^{\infty}$ spectral triple and let $c\in\mathcal{A}^{\otimes (p+1)}$ be a Hochschild cycle.
\begin{enumerate}[{\rm (a)}]
\item\label{k5a} If $D^{-p}\in\mathcal{L}_{1,\infty},$ then
$\mathcal{W}_{\varnothing}(c)\in[\mathcal{L}_{1,\infty},\mathcal{L}(H)].$
\item\label{k5b} If $D^{-p}\in\mathcal{M}_{1,\infty},$ then
$
\mathcal{W}_{\varnothing}(c)\in[\mathcal{M}_{1,\infty},\mathcal{L}(H)].
$
\end{enumerate}
\end{lem}
\begin{proof} We prove \eqref{k5a} only (the proof of \eqref{k5b} is identical). Let $a_0\otimes\cdots\otimes a_p\in\mathcal{A}^{\otimes (p+1)}.$ We have
\begin{equation}
2\Gamma a_0\prod_{k=1}^p[F,a_k]
=[F,F\Gamma a_0\prod_{k=1}^p[F,a_k]]+(-1)^{p-1}F\Gamma\prod_{k=0}^p[F,a_k]
\label{eq:bob1}
\end{equation}
so that 
\begin{equation}
2\mathcal{W}_{\varnothing}(c)=[F,F\mathcal{W}_{\varnothing}(c)]+(-1)^{p-1}ch(c).
\label{eq:bob2}
\end{equation}
Since $\mathcal{W}_{\varnothing}(c)\in\mathcal{L}_{1,\infty},$ 
it follows that $[F,F\mathcal{W}_{\varnothing}(c)]\in[\mathcal{L}_{1,\infty},\mathcal{L}(H)]$. 
By Proposition \ref{f der def} and the H\"older property in Equation \eqref{lpi mult}, we have 
$ch(c)\in\mathcal{L}_1\subset[\mathcal{L}_{1,\infty},\mathcal{L}(H)].$ 
Thus, $\mathcal{W}_{\varnothing}(c)\in[\mathcal{L}_{1,\infty},\mathcal{L}(H)].$
\end{proof}

We are now ready to prove the main result of this subsection.

\begin{proof}[Proof of Proposition \ref{reduction}] As in preceding lemma, we prove \eqref{reda} only (the proof of \eqref{redb} is identical). For every Hochschild cycle $c\in\mathcal{A}^{\otimes (p+1)},$ it follows from Lemma \ref{kogom6} that
$$
\Omega(c)|D|^{-p}\in\sum_{\mathscr{A}
\subset \{1,\dots,p\}}(-1)^{n_{\mathscr{A}}}\mathcal{W}_{\mathscr{A}}(c)D^{-|\mathscr{A}|}
+\mathcal{L}_1.
$$
Applying Corollary \ref{kogom4} to every summand in the sum $\sum_{|\mathscr{A}|\geq 2}$ and Lemma \ref{kogom5}, we infer that
$$
\Omega(c)|D|^{-p}\in\sum_{|\mathscr{A}|=1}
(-1)^{n_{\mathscr{A}}}\mathcal{W}_{\mathscr{A}}(c)D^{-1}+[\mathcal{L}_{1,\infty},\mathcal{L}(H)].
$$
If $\mathscr{A}=\{m\},$ then $n_{\mathscr{A}}=p-m.$ Therefore,
$$
\Omega(c)|D|^{-p}\in\sum_{m=1}^p(-1)^{p-m}\mathcal{W}_m(c)D^{-1}
+[\mathcal{L}_{1,\infty},\mathcal{L}(H)].
$$
Applying Lemma \ref{kogom3} $p-m$ times, we obtain
$$
\mathcal{W}_m(c)D^{-1}-(-1)^{p-m}\mathcal{W}_p(c)D^{-1}
\in[\mathcal{L}_{1,\infty},\mathcal{L}(H)],\quad 1\leq m<p.
$$
This suffices to conclude the proof.
\end{proof}

\subsection{Some commutator estimates}\label{commutator section}

Our method of proof of  Proposition \ref{second cycle lemma} 
exploits some heat semigroup asymptotics. For this we need to 
introduce, in this subsection, a number of technical estimates for 
commutators involving the operator valued function 
$s\to f(s|D|),$ where $f(s)=e^{-|s|^{p+1}},$
and  $s\in\mathbb{R}.$  As before in the text, $p\in\mathbb{N}.$ 
We make essential use of the fact that $\hat{f''}\in L_1(-\infty,\infty)$ 
(this fact follows from Lemma 7 in \cite{PS}).

\begin{lem}\label{first commutator lemma} If $(\mathcal{A},H,D)$ is a
$QC^{\infty}$ spectral triple, then
$$
\|[f(s|D|),a]-sf'(s|D|)\delta(a)\|_{\infty}\leq s^2\|\hat{f''}\|_1\|\delta^2(a)\|_{\infty}
$$
$$
\|[f(s|D|),a]-s\delta(a)f'(s|D|)\|_{\infty}\leq s^2\|\hat{f''}\|_1\|\delta^2(a)\|_{\infty}
$$
for all $s>0$ and for all $a\in\mathcal{A}.$
\end{lem}
\begin{proof} We use the method of \cite{BeF,CPRS}. It is clear that
\begin{equation}\label{ell1}
[f(s|D|),a]=\int_{-\infty}^{\infty}\hat{f}(u)[e^{ius|D|},a]du.
\end{equation}
An elementary computation shows that
\begin{equation}\label{ell2}
[e^{ius|D|},a]=ius\int_0^1e^{iuvs|D|}\delta(a)e^{iu(1-v)s|D|}dv.
\end{equation}
Combining \eqref{ell1} and \eqref{ell2}, we obtain
$$
[f(s|D|),a]=s\int_{-\infty}^{\infty}\int_0^1\hat{f'}(u)e^{iuvs|D|}\delta(a)e^{iu(1-v)s|D|}dvdu.
$$
Therefore,
\begin{align*}
[f(s|D|),a]-sf'(s|D|)\delta(a)&
=s\int_{-\infty}^{\infty}\int_0^1
\hat{f'}(u)\Big(e^{iuvs|D|}\delta(a)e^{iu(1-v)s|D|}-e^{ius|D|}\delta(a)\Big)dvdu\\
&=s\int_{-\infty}^{\infty}\int_0^1\hat{f'}(u)\Big(e^{iuvs|D|}[\delta(a),e^{iu(1-v)s|D|}]\Big)dvdu.
\end{align*}
As in Equation \eqref{ell2}, we have
$$
[\delta(a),e^{iu(1-v)s|D|}]=-iu(1-v)s\int_0^1e^{iu(1-v)sw|D|}\delta^2(a)e^{iu(1-v)s(1-w)|D|}dw.
$$
Hence,
$$
[f(s|D|),a]-sf'(s|D|)\delta(a)=-s^2\int_{-\infty}^{\infty}\int_0^1\int_0^1\hat{f''}(u)(1-v)e^{iu(1-v)sw|D|}\delta^2(a)e^{iu(1-v)s(1-w)|D|}dwdvdu.
$$
The first inequality follows immediately. The proof of the second inequality is similar so we omit it.
\end{proof}

\begin{lem}\label{eddl} Let $D$ be an invertible unbounded self-adjoint operator.
\begin{enumerate}[{\rm (a)}]
\item\label{eddla} If $D^{-p}\in\mathcal{L}_{1,\infty},$ then
$
{\rm Tr}(f(s|D|))=O(s^{-p}),\quad {\rm Tr}(|D|^{-p-1}(1-f(s|D|)))=O(s),\quad s\to0.
$
\item\label{eddlb} If $D^{-p}\in\mathcal{M}_{1,\infty},$ then (for every $\varepsilon>0$)
$$
{\rm Tr}(f(s|D|))=O(s^{-p-\varepsilon}),\quad {\rm Tr}(|D|^{-p-1}(1-f(s|D|)))
=O(s^{1-\varepsilon}),\quad s\to0.
$$
\end{enumerate}
\end{lem}
\begin{proof} Using Lemma \ref{estimate} with $V=|D|^{-p}$ and $\alpha=1+1/p,$ we obtain \eqref{eddla}. We now prove \eqref{eddlb}. Since $D^{-p}\in\mathcal{M}_{1,\infty},$ it follows that
$$
(k+1)\mu(k,D^{-p})\leq\sum_{m=0}^k\mu(m,D^{-p})\leq {\rm const}\cdot\log(k+2).
$$
Hence,
$$
\mu(k,D^{-p-\varepsilon})\leq
\left({\rm const}\cdot\frac{\log(k+2)}{k+1}\right)^{(p+\varepsilon)/p}
\leq\frac{{\rm const}}{k+1},\quad k\geq0.
$$
Select an operator $D_0\leq D$ (using the same eigenbasis) such that
$
\mu(k,D_0^{-p-\varepsilon})=\frac{{\rm const}}{k+1},\  k\geq0.
$
In what follows, we assume, to reduce the notation, that ${\rm const}=1.$
For the first equality we have:
$$
{\rm Tr}(f(s|D|))\leq{\rm Tr}(f(s|D_0|))=\sum_{k=1}^{\infty}e^{-(sk^{1/(p+\varepsilon)})^{p+1}}
\leq\int_0^{\infty}e^{-(su^{1/(p+\varepsilon})^{p+1}}du=s^{-p-\varepsilon}\int_0^{\infty}e^{-v^{(p+1)/(p+\varepsilon)}}dv.
$$

In order to prove the second equality, note that the mapping 
$s\to s^{-1}(1-e^{-s})$ is decreasing on $(0,\infty)$ and so is the mapping $s\to s^{-p-1}(1-f(s)).$ It follows that
\begin{align*}
{\rm Tr}(&|D|^{-p-1}(1-f(s|D|)))\leq {\rm Tr}(|D_0|^{-p-1}(1-f(s|D_0|)))
=\sum_{k=1}^{\infty}(k^{1/(p+\varepsilon)})^{-p-1}(1-e^{-(sk^{1/(p+\varepsilon)})^{p+1}})\\
&\leq\int_0^{\infty}u^{-(p+1)/(p+\varepsilon)}(1-e^{-(su^{1/(p+\varepsilon)})^{p+1}})du
=s^{1-\varepsilon}\int_0^{\infty}v^{-(p+1)/(p+\varepsilon)}(1-e^{-v^{(p+1)/(p+\varepsilon)}})dv.
\end{align*}
\end{proof}

\begin{lem}\label{gr antik} Let $(\mathcal{A},H,D)$ be a $QC^{\infty}$ spectral triple and let $a\in\mathcal{A}.$
\begin{enumerate}[{\rm (a)}]
\item If $D^{-p}\in\mathcal{L}_{1,\infty},$ then \quad
$\|[f'(s|D|),\delta(a)]\|_1=O(s^{1-p})$ as $s\to0$.
\item If $D^{-p}\in\mathcal{M}_{1,\infty}$ then (for every $\varepsilon>0$)\quad
$\|[f'(s|D|),\delta(a)]\|_1=O(s^{1-p-\varepsilon})$, as $s\to0$.
\end{enumerate}
\end{lem}
\begin{proof} Suppose first that $p\geq 4$ or that $p=2.$ Define a positive function $h$ by setting $f'(t)=-{\rm sgn}(t)h^2(t)$ for all $t.$ We have $h',h''\in L_2(-\infty,\infty).$ It follows now from Lemma 7 in \cite{PS} that $\hat{h'}\in L_1(-\infty,\infty).$ Repeating the argument in the beginning of Lemma \ref{first commutator lemma}, we obtain
$$[h(s|D|),\delta(a)]=s\int_{-\infty}^{\infty}\int_0^1\hat{h'}(u)e^{iuvs|D|}\delta^2(a)e^{iu(1-v)s|D|}dvdu$$
and, therefore,
$\|[h(s|D|),\delta(a)]\|_{\infty}\leq s\|\hat{h'}\|_1\|\delta^2(a)\|_{\infty}.$
On the other hand, we have
$$[f'(s|D|,\delta(a)]=[h^2(s|D|),\delta(a)]=h(s|D|)[h(s|D|),\delta(a)]+[h(s|D|),\delta(a)]h(s|D|).$$
Therefore,
$$\|[f'(s|D|),\delta(a)]\|_1\leq 2\|h(s|D|)\|_1\|[h(s|D|),\delta(a)]\|_{\infty}=\|h(s|D|)\|_1\cdot O(s).$$

Recall that $h(s)\leq{\rm const}\cdot f(s/2)$ for all $s\in\mathbb{R}.$ If $D^{-p}\in\mathcal{L}_{1,\infty},$ then it follows from Lemma \ref{eddl} \eqref{eddla} that $\|h(s|D|)\|_1=O(s^{-p}).$ Similarly, if $D^{-p}\in\mathcal{M}_{1,\infty},$ then it follows from Lemma \ref{eddl} \eqref{eddlb} that $\|h(s|D|)\|_1=O(s^{-p-\varepsilon}).$ This proves the assertion for $p\geq 4$ or $p=2.$

If $p=1$ or $p=3$, then Lemma 7 in \cite{PS} is inapplicable and 
we have to proceed with a direct computation. Assume, for simplicity, 
that $p=1$ and $D^{-1}\in\mathcal{L}_{1,\infty}$ (the proof is similar for 
$p=3$ and for $\mathcal{M}_{1,\infty}$). Repeating the argument above, we obtain
$$
\|[f^{1/2}(s|D|),\delta(a)]\|_1=O(s),\quad \|[f^{1/2}(s|D|),\delta^2(a)]\|_1=O(s).
$$
Using the elementary equality
\begin{align*}
-\frac12[f'(s|D|),\delta(a)]&=\delta^2(a)\cdot sf(s|D|)+s|D|f^{1/2}(s|D|)\cdot[f^{1/2}(s|D|),\delta(a)]\\
&+[f^{1/2}(s|D|),\delta(a)]\cdot s|D|f^{1/2}(s|D|)+[f^{1/2}(s|D|),\delta^2(a)]\cdot sf^{1/2}(s|D|),
\end{align*}
we infer that
$$
\|[f'(s|D|),\delta(a)]\|_1\leq{\rm const}\cdot{\rm Tr}(sf(s|D|)+2s^2|D|f^{1/2}(s|D|)+s^2f^{1/2}(s|D|)).
$$

Recall that $sf(s),f^{1/2}(s)\leq{\rm const}\cdot f(s/2)$ for all $s>0.$ By Lemma \ref{eddl} \eqref{eddla}, we have
$$
s{\rm Tr}(f(s|D|))=O(1),\quad s{\rm Tr}(s|D|f^{1/2}(s|D|))=O(1),\quad s{\rm Tr}(f^{1/2}(s|D|))=O(1).
$$
This proves the assertion for $p=1.$
\end{proof}

\begin{lem}\label{second commutator lemma} 
Let $(\mathcal{A},H,D)$ be a $QC^{\infty}$ spectral triple and let $a\in\mathcal{A}.$
\begin{enumerate}[{\rm (a)}]
\item\label{seccoma} If $D^{-p}\in\mathcal{L}_{1,\infty},$ then \quad
$\|[f(s|D|),a]-s\delta(a)f'(s|D|)\|_1=O(s^{2-p})$, as $s\to0.$
\item\label{seccomb} If $D^{-p}\in\mathcal{M}_{1,\infty},$ then (for every $\varepsilon>0$)\quad
$\|[f(s|D|),a]-s\delta(a)f'(s|D|)\|_1=O(s^{2-p-\varepsilon})$, as $s\to0$.
\end{enumerate}
\end{lem}
\begin{proof} Let $f=h^2.$ Since $h$ can be obtained from $f$ by rescaling, the assertion of Lemma \ref{first commutator lemma} also holds for $h.$ We have
$$[f(s|D|,a)]-\frac{s}2\{f'(s|D|),\delta(a)\}=h(s|D|)\Big([h(s|D|),a]-sh'(s|D|)\delta(a)\Big)$$
$$\quad\quad\quad\quad\quad\quad\quad\quad\quad\quad\quad\quad\quad\quad+\Big([h(s|D|),a]-s\delta(a)h'(s|D|)\Big)h(s|D|).$$
It follows that
$$\|[f(s|D|,a)]-\frac{s}2\{f'(s|D|),\delta(a)\}\|_1$$
$$\leq\|h(s|D|)\|_1\Big(\|[h(s|D|),a]-sh'(s|D|)\delta(a)\|_{\infty}+\|[h(s|D|),a]-s\delta(a)h'(s|D|)\|_{\infty}\Big).$$
We infer from Lemma \ref{first commutator lemma} that the expression in brackets is $O(s^2).$
If $D^{-p}\in\mathcal{L}_{1,\infty},$ then it follows from Lemma \ref{eddl} \eqref{eddla} that $\|h(s|D|)\|_1=O(s^{-p}).$ Therefore,
$$\|[f(s|D|,a)]-\frac{s}2\{f'(s|D|),\delta(a)\}\|_1=O(s^{2-p}).$$
The assertion \eqref{seccoma} follows now from Lemma \ref{gr antik}. Similarly, if $D^{-p}\in\mathcal{M}_{1,\infty},$ then it follows from Lemma \ref{eddl} \eqref{eddlb} that $\|h(s|D|)\|_1=O(s^{-p-\varepsilon}).$ This proves the assertion \eqref{seccomb}.
\end{proof}

\begin{prop}\label{commutator estimate} 
Let $(\mathcal{A},H,D)$ be a $QC^{\infty}$ spectral triple and let $a\in\mathcal{A}.$
\begin{enumerate}[{\rm (a)}]
\item If $D^{-p}\in\mathcal{L}_{1,\infty},$ then \quad
$\|[f(s|D|),a]-s\delta(a)f'(s|D|)\|_{p,1}=O(s)$, as $s\to0.$
\item If $D^{-p}\in\mathcal{M}_{1,\infty},$ then (for every $\varepsilon>0$) \quad
$\|[f(s|D|),a]-s\delta(a)f'(s|D|)\|_{p-\varepsilon}=O(s^{1-2\varepsilon})$, as $s\to0.$
\end{enumerate}
\end{prop}
\begin{proof} We prove only the first assertion, as the proof of the second one is identical. 
If $p=1,$ then the assertion is proved in Lemma \ref{second commutator lemma}. 
Suppose $p>1$ and set
$$
T=[f(s|D|),a]-s\delta(a)f'(s|D|).
$$
We infer from Lemma \ref{first commutator lemma}  
that $\|T\|_{\infty}=O(s^2)$ and from Lemma \ref{second commutator lemma} that 
$\|T\|_1=O(s^{2-p})$ as $s\to0.$ The assertion follows from the 
interpolation inequality (see e.g. Theorem 2.g.18 and Corollary 2.g.14 in \cite{LT2})
$$
\|T\|_{p,1}\leq\|T\|_1^{1/p}\|T\|_{\infty}^{1-1/p}=O(s^{(2-p)/p}\cdot s^{2(1-1/p)})=O(s).
$$
\end{proof}

\subsection{Asymptotics for the heat semigroup and the 
proof of Proposition \ref{second cycle lemma}}\label{heat}

In order to study the operator $\mathcal{W}_p(c)D^{-1}$, which was introduced 
in Proposition \ref{reduction}, we now establish the following heat semigroup estimate.

\begin{prop}\label{second cycle lemma} Let $(\mathcal{A},H,D)$ be a $QC^{\infty}$ 
spectral triple with $D^{-p}\in\mathcal{M}_{1,\infty}.$ If the spectral triple and the 
integer $p$ are both odd (respectively, even), then
$$
{\rm Tr}(\mathcal{W}_p(c)D^{-1}e^{-(s|D|)^{p+1}})
=Ch(c)\log\left(1/s\right)+O(1),\quad s\to0,
$$
for every Hochschild cycle $c\in\mathcal{A}^{\otimes (p+1)}.$
\end{prop}

In 
Lemma \ref{f second est} and Lemma \ref{first cycle lemma}, 
we prepare the ground for the proof of Proposition \ref{second cycle lemma}.

\begin{lem}\label{f second est} If $(\mathcal{A},H,D)$ is a $QC^{\infty}$ spectral triple, then
$$
\left(\prod_{k=0}^m[F,a_k]\right)|D|^{m+1}\in\mathcal{L}(H),\quad a_k\in\mathcal{A},\quad 0\leq k\leq m.
$$
\end{lem}
\begin{proof}Define the algebra 
$\mathcal{B}=\{A\in\mathcal{L}(H):\ A:{\rm dom}(D)\to{\rm dom}(D),
\ \delta^n(A)\in\mathcal{L}(H)\mbox{ for all }n\geq0\}.$
An inductive argument shows that, for every $A\in\mathcal{B}$ and 
for every $n\geq0,$ there exists $B\in\mathcal{B}$ such that $A|D|^n=|D|^nB.$
For all $k\leq m$ and for all $a_k\in\mathcal{A},$ we have $[D,a_k]\in\mathcal{B}$ 
and $F[|D|,a_k]\in\mathcal{B}$ (here, we used the fact that our spectral triple is 
$QC^{\infty}$). Therefore,
$$
[F,a_k]=[D,a_k]|D|^{-1}-F[|D|,a_k]|D|^{-1}=A_k|D|^{-1},
$$
where $A_k\in\mathcal{B}.$ Therefore,
$$
\prod_{k=0}^m[F,a_k]|D|^{m+1}=\left(\prod_{k=0}^{m-1}[F,a_k]\right)A_m|D|^m.
$$
Note that $A_m|D|^{-1}\cdot|D|^{m+1}=|D|^mB_m$ for some $B_m\in\mathcal{B}.$ It follows that
$$
\prod_{k=0}^m[F,a_k]|D|^{m+1}=\left(\prod_{k=0}^{m-1}[F,a_k]\right)|D|^mB_m.
$$
The right hand side is bounded by induction.
\end{proof}

Note that the condition $D^{-p}\in\mathcal{M}_{1,\infty}$ guarantees that $D^{-p-2}\in\mathcal{L}_1.$ Hence,
$$0\leq -f'(s|D|)\leq\frac{4(p+1)}{e}(s|D|)^{-p-2}\in\mathcal{L}_1.$$
In particular, we have $f'(s|D|)\in\mathcal{L}_1.$

\begin{lem}\label{first cycle lemma} Let $(\mathcal{A},H,D)$ be a $QC^{\infty}$ spectral triple and let $c\in\mathcal{A}^{\otimes (p+1)}$ be a Hochschild cycle. Suppose that the spectral triple and $p$ are both odd (respectively, even).
\begin{enumerate}[{\rm (a)}]
\item If $D^{-p}\in\mathcal{L}_{1,\infty},$ then \quad
$
s{\rm Tr}(\mathcal{W}_p(c)Ff'(s|D|))=-Ch(c)+O(s)$, as $s\to0.
$
\item If $D^{-p}\in\mathcal{M}_{1,\infty},$ then (for every $\varepsilon>0$) \quad
$
s{\rm Tr}(\mathcal{W}_p(c)Ff'(s|D|))=-Ch(c)+O(s^{1-\varepsilon})$, as $s\to0.
$
\end{enumerate}
Here, $f(s)=e^{-|s|^{p+1}},$ $s\in\mathbb{R}.$
\end{lem}
\begin{proof} We only prove the first assertion. The proof of the second one is identical. 
Define the multilinear mappings $\mathcal{K}_s,\,\mathcal{H}_s:\mathcal{A}^{\otimes (p+1)}\to\mathcal{L}(H)$ by setting
$$
\mathcal{K}_s(a_0\otimes\cdots\otimes a_p)
=\Gamma a_0\Big(\prod_{k=1}^{p-1}[F,a_k]\Big)[Ff(s|D|),a_p],
\quad\ \ 
\mathcal{H}_s(a_0\otimes\cdots\otimes a_p)
=\Gamma a_0\Big(\prod_{k=1}^{p-1}[F,a_k]\Big)F[f(s|D|),a_p].
$$
For all $c\in\mathcal{A}^{\otimes (p+1)},$ we have (see p. 293 in \cite{Connes} for the second equality)
$$
\mathcal{W}_{\varnothing}(c)f(s|D|)=\mathcal{K}_s(c)-\mathcal{H}_s(c),
\quad ch(c)=\mathcal{W}_{\varnothing}(c)+F\mathcal{W}_{\varnothing}(c)F.
$$
Therefore,
\begin{equation}\label{qsq1}
{\rm Tr}(ch(c)f(s|D|))=2{\rm Tr}(\mathcal{W}_{\varnothing}(c)f(s|D|))
=2{\rm Tr}(\mathcal{K}_s(c))-2{\rm Tr}(\mathcal{H}_s(c)).
\end{equation}

The mapping $c'\to{\rm Tr}(\mathcal{K}_s(c'))$ on $\mathcal{A}^{\otimes (p+1)}$ 
is the Hochschild coboundary\footnote{For the sake of illustration, let $p=2$ and let the multilinear mapping $\theta:\mathcal{A}^{\otimes 2}\to\mathcal{L}(H)$ be defined by setting
$$
\theta(a_0\otimes a_1)={\rm Tr}(\Gamma a_0[F,a_1]T)
$$
with $T\in\mathcal{L}_1.$ We then have
\begin{align*}
(b\theta)(a_0\otimes a_1\otimes a_2)&=\theta(a_0a_1\otimes a_2)-\theta(a_0\otimes a_1a_2)+\theta(a_2a_0\otimes a_1)\\
&={\rm Tr}(\Gamma a_0a_1[F,a_2]T-\Gamma a_0[F,a_1a_2]T)+{\rm Tr}(\Gamma a_2a_0[F,a_1]T)\\&=-{\rm Tr}(\Gamma a_0[F,a_1]a_2T)+{\rm Tr}(\Gamma a_0[F,a_1]Ta_2)={\rm Tr}(\Gamma a_0[F,a_1][T,a_2]).
\end{align*}} 
of the multilinear mapping defined by 
$$a_0\otimes\cdots\otimes a_{p-1}\to(-1)^p{\rm Tr}(\Gamma a_0\Big(\prod_{k=1}^{p-1}[F,a_k]\Big)Ff(s|D|)).$$
Hence, it vanishes on every Hochschild cycle.
On the other hand, we have
\begin{equation}\label{qsq2}
{\rm Tr}(\mathcal{H}_s(c'))=s{\rm Tr}(\mathcal{W}_p(c')Ff'(s|D|))+O(s)
\end{equation}
as may be seen by evaluating on $a_0\otimes\cdots\otimes a_p,$ using Proposition \ref{commutator estimate} to obtain
\begin{align*}
&\left|{\rm Tr}(\Gamma a_0\prod_{k=1}^{p-1}[F,a_k]F[f(s|D|),a_p])-s{\rm Tr}(\Gamma a_0\prod_{k=1}^{p-1}[F,a_k]F\delta(a_p)f'(s|D|))\right|\\
&\leq\|\Gamma a_0\prod_{k=1}^{p-1}[F,a_k]F\|_{q,\infty}\,\|[f(s|D|),a_p]-s\delta(a_p)f'(s|D|)\|_{p,1}=O(s)
\end{align*}
and, since,
\begin{align*}
&\left|{\rm Tr}(\Gamma a_0\prod_{k=1}^{p-1}[F,a_k]F\delta(a_p)f'(s|D|))-{\rm Tr}(\mathcal{W}_p(a)Ff'(s|D|))\right|\\
&\leq\|(\Gamma a_0\Big(\prod_{k=1}^{p-1}[F,a_k]\Big)[F,\delta(a_p)]|D|^p\|_{\infty}\cdot\||D|^{-p}f'(s|D|)\|_1=O(1),
\end{align*}
the equality \eqref{qsq2} follows.
Combining the equalities 
\eqref{qsq1}, \eqref{qsq2} and the fact that 
${\rm Tr}(\mathcal{K}_s(c))=0$ for every Hochschild 
cycle $c\in\mathcal{A}^{\otimes (p+1)},$ we infer that
\begin{equation}\label{chs1}
{\rm Tr}(ch(c)f(s|D|))=-2s{\rm Tr}(\mathcal{W}_p(c)Ff'(s|D|))+O(s)
\end{equation}
for every Hochschild cycle $c\in\mathcal{A}^{\otimes (p+1)}.$
The operator $B=ch(c)|D|^{p+1}$ is bounded by Lemma \ref{f second est}. 
Using Lemma \ref{eddl} \eqref{eddla}, we obtain
$$
|{\rm Tr}(B{f(s|D|)}{|D|^{-p-1}})-{\rm Tr}(B{|D|^{-p-1}})|\leq\|B\|_{\infty}{\rm Tr}(({1-f(s|D|)}){|D|^{-p-1}})=O(s).
$$
Therefore,
\begin{equation}\label{chs2}
{\rm Tr}(ch(c)f(s|D|))=Ch(c)+O(s).
\end{equation}
By combining \eqref{chs1} and \eqref{chs2}, we conclude the proof.
\end{proof}

\begin{proof}[Proof of Proposition \ref{second cycle lemma}] By Lemma \ref{first cycle lemma}, we have
$$
{\rm Tr}(\mathcal{W}_p(c)F|D|^pe^{-(s|D|)^{p+1}})
=\frac{1}{(p+1)}Ch(c)s^{-p-1}+O(s^{\varepsilon-p}).
$$
Setting $u=s^{p+1},$ we obtain
$$
{\rm Tr}(\mathcal{W}_p(c)F|D|^pe^{-u|D|^{p+1}})
=\frac{1}{(p+1)u}Ch(c)+O(u^{-(p-\varepsilon)/(p+1)}).
$$
Integrating over $u\in[s,1],$ we obtain
$$
{\rm Tr}(\mathcal{W}_p(c)F|D|^{-1}(e^{-s|D|^{p+1}}-e^{-|D|^{p+1}}))
=\frac{1}{(p+1)}Ch(c)\log\left(\frac1s\right)+O(1).
$$
Taking into account that $D^{-p}\in\mathcal{M}_{1,\infty}$ implies that
$
\mathcal{W}_p(c)F|D|^{-1}e^{-|D|^{p+1}}\in\mathcal{L}_1.
$
Replacing $s$ with $s^{p+1},$ we conclude the proof.
\end{proof}

\subsection{Proof of the main result}\label{prmain}

In this subsection, we prove Theorem \ref{main result}. 
Recall that the multilinear mapping $\mathcal{W}_p$ is defined in Section \ref{plan}.

\begin{lem}\label{intermed theorem} Let $(\mathcal{A},H,D)$ be an  odd (respectively, even) 
$QC^{\infty}$ spectral triple and let $c\in\mathcal{A}^{\otimes (p+1)}$ be a 
Hochschild cycle. Suppose that $p$ is odd (respectively, even).
\begin{enumerate}[{\rm (a)}]
\item If $D^{-p}\in\mathcal{L}_{1,\infty},$ then 
$
\varphi(\mathcal{W}_p(c)D^{-1})=\frac{Ch(c)}{p}
$
for every normalised trace $\varphi$ on $\mathcal{L}_{1,\infty}.$
\item If $D^{-p}\in\mathcal{M}_{1,\infty},$ then 
$
{\rm Tr}_{\omega}(\mathcal{W}_p(c)D^{-1})=\frac{Ch(c)}{p}
$
for every Dixmier trace on $\mathcal{M}_{1,\infty}.$
\end{enumerate}
\end{lem}
\begin{proof} Recall the algebra 
$
\mathcal{B}
=\{A\in\mathcal{L}(H):\ A:{\rm dom}(D)\to{\rm dom}(D),\ \delta^n(A)\in\mathcal{L}(H)\mbox{ for all }n\geq0\}.
$
It follows from Lemma \ref{f second est} that $\mathcal{W}_p(a)|D|^{p-1}\in\mathcal{B}$ and is, therefore, bounded. 
Set $V=|D|^{-p}$ and $\alpha=1+1/p.$ It follows from Proposition \ref{second cycle lemma} that
$$
{\rm Tr}(\mathcal{W}_p(c)D^{-1}e^{-(nV)^{-\alpha}})=\frac{Ch(c)}{p}\log(n)+O(1)
$$
as $n\to\infty.$ By the previous paragraph, we have 
$A=\mathcal{W}_p(c)F|D|^{p-1}\in\mathcal{L}(H)$ and, 
by assumption, $V\in\mathcal{L}_{1,\infty}$ (respectively, 
$V\in\mathcal{M}_{1,\infty}$). Therefore, Proposition 
\ref{measurability criterion} is applicable and yields
$
\varphi(\mathcal{W}_p(c)D^{-1})=\frac{Ch(c)}{p}
$
for every normalised trace $\varphi$ on $\mathcal{L}_{1,\infty}$ or
$
{\rm Tr}_{\omega}(\mathcal{W}_p(c)D^{-1})=\frac{Ch(c)}{p}
$
for every Dixmier trace on $\mathcal{M}_{1,\infty},$ respectively.
\end{proof}

\begin{lem}\label{better triple} If $(\mathcal{A},H,D)$ is a $QC^{\infty}$ spectral triple, then so is $(\mathcal{A},H,D_0),$ where $D_0=F(1+D^2)^{1/2}.$
\end{lem}
\begin{proof} Set $D_1=D_0-D\in\mathcal{L}(H).$ Define the operations $\delta_0:a\to[|D_0|,a]$ and $\delta_1:a\to[|D_1|,a].$ Noting that $|D_0|=|D|+|D_1|,$ we infer that $\delta_0=\delta+\delta_1.$ Since the operations $\delta_0$ and $\delta_1$ commute, it follows that
$$\delta_0^n(a)=\sum_{k=0}^n\binom{n}{k}\delta_1^{n-k}(\delta^k(a)).$$
Since $\delta^k(a)$ is well defined and since $\delta_1:\mathcal{L}(H)\to\mathcal{L}(H)$ is a bounded mapping, it follows that $\delta_0^n(a)$ is well defined.
Similarly, $\delta_0^n(\partial(a))$ is well defined. Define the operations $\partial_0:a\to[D_0,a]$ and $\partial_1:a\to[D_1,a].$ We have
$$\delta_0^n(\partial_0(a))=\delta_0^n(\partial(a))+\delta_0^n(\partial_1(a))=\delta_0^n(\partial(a))+\partial_1(\delta_0^n(a)).$$
By Definition \ref{qc}, $(\mathcal{A},H,D_0)$ is a $QC^{\infty}$ spectral triple.
\end{proof}

We are now ready to prove the main result of the paper. 
We present the detailed argument for the first part of the theorem.

{\bf Case 1:} Suppose that $(\mathcal{A},H,D)$ is a $QC^{\infty}$ 
odd $(p,\infty)-$summable spectral triple and that $p$ is even.

Let $\varphi$ be a trace on $\mathcal{L}_{1,\infty}.$ The mapping defined on
$\mathcal{A}^{\otimes (p+1)}$ by
$c\to\varphi(\Omega(c)(1+D^2)^{-p/2})$
is the Hochschild coboundary (see Appendix \ref{cobos}) of the multilinear mapping
defined by
$$
a_0\otimes\cdots\otimes a_{p-1}\to\frac12\varphi\Big(\prod_{k=0}^{p-1}[D,a_k](1+D^2)^{-p/2}\Big).
$$
Every Hochschild coboundary vanishes on every Hochschild cycle, so that 
$\varphi(\Omega(c)(1+D^2)^{-p/2})=0$ for every Hochschild cycle 
$c\in\mathcal{A}^{\otimes (p+1)}.$ Thus, the left hand side of \eqref{main equation} 
vanishes.
For $c'=a_0\otimes\cdots\otimes a_p,$ with $p$  even, 
$F\prod_{k=0}^p[F,a_k]=-\prod_{k=0}^p[F,a_k]F$ and, therefore,
$$
Ch(c')={\rm Tr}\Big(F\prod_{k=0}^p[F,a_k]\Big)=-{\rm Tr}\Big(\prod_{k=0}^p[F,a_k]F\Big)=-Ch(c').
$$
Hence, $Ch(c')=0$ for all $c'\in\mathcal{A}^{\otimes (p+1)}.$ 
Thus, the right hand side of \eqref{main equation} vanishes.

{\bf Case 2:} Suppose that $(\mathcal{A},H,D)$ is a $QC^{\infty}$ even $(p,\infty)-$summable spectral triple and that $p$ is odd.

Let $\varphi$ be a trace on $\mathcal{L}_{1,\infty}.$ By Definition \ref{oddeven}, we have $\Gamma[D,a]=-[D,a]\Gamma$ and $\Gamma a=a\Gamma$ for all $a\in\mathcal{A}.$ Since $p$ is odd, it follows that
$$\Gamma a_0\prod_{k=1}^p[D,a_k](1+D^2)^{-p/2}=a_0\Gamma\prod_{k=1}^p[D,a_k](1+D^2)^{-p/2}$$
$$\quad\quad\quad=-a_0\prod_{k=1}^p[D,a_k]\Gamma(1+D^2)^{-p/2}=-a_0\prod_{k=1}^p[D,a_k](1+D^2)^{-p/2}\Gamma.$$
Applying the trace $\varphi,$ we obtain
$$
\varphi\Big(\Gamma a_0\prod_{k=1}^p[D,a_k](1+D^2)^{-p/2}\Big)
=-\varphi\Big(\Gamma a_0\prod_{k=1}^p[D,a_k](1+D^2)^{-p/2}\Big).
$$
Hence, the left hand side of \eqref{main equation} vanishes.
Repeating the argument in Step 1, we infer that the right hand side of \eqref{main equation} vanishes as well.

{\bf Case 3:} Suppose that $p$ and the $(p,\infty)-$summable spectral triple $(\mathcal{A},H,D)$ are simultaneously odd (or even).

If $D$ is invertible, then we infer from Proposition \ref{reduction} and Lemma \ref{intermed theorem} that
$$
\varphi(\Omega(c)|D|^{-p})=p\varphi(\mathcal{W}_p(c)D^{-1})=Ch(c)
$$
and the assertion is proved.
Suppose now that $D$ is not invertible and consider the invertible operator 
$D_0=F(1+D^2)^{1/2}.$ It follows from Lemma \ref{better triple} that $(\mathcal{A},H,D_0)$ is a spectral triple with $D_0^{-p}\in\mathcal{L}_{1,\infty}.$ Clearly,
$$
D_1:=D_0-D=F((1+|D|^2)^{1/2}-|D|)=F(|D|+(1+|D|^2)^{1/2})^{-1}\in\mathcal{L}_{p,\infty}.
$$
We claim that 
\begin{equation}\label{claim}
a_0\prod_{k=1}^p[D,a_k]|D_0|^{-p}-a_0\prod_{k=1}^p[D_0,a_k]|D_0|^{-p}\in\mathcal{L}_1
\end{equation}
for $a_0\otimes\cdots\otimes a_p\in\mathcal{A}^{\otimes (p+1)}.$ To see the claim, let us write
$$
\prod_{k=1}^p[D_0,a_k]=\sum _{\mathscr{A}\subset \{1,2,\dots,p\}}\prod_{k=1}^p\begin{cases} [D, a_k]  &\text{for } k\in \mathscr{A}\\
[D_1, a_k]  &\text{for } k\notin \mathscr{A} \end{cases}
$$
The summand corresponding to the case $\mathscr{A}=\{1,2,\dots,p\}$ coincides with $a_0\prod_{k=1}^p[D,a_k]|D_0|^{-p}$, while all other summands belong to $\mathcal{L}_1$. Indeed, since there exists $k\notin\mathscr{A}$, it follows that the product contains the term $[D_1, a_k]\in \mathcal{L}_{p,\infty}$. Thus, such a summand belongs to 
$\mathcal{L}_{p,\infty}\cdot \mathcal{L}_{1,\infty}\subset \mathcal{L}_{1}$ (by Equation \eqref{lpi mult}).
Since the assertion holds for the spectral triple 
$(\mathcal{A},H,D_0),$ we infer that it also holds 
for the spectral triple $(\mathcal{A},H,D).$

{\bf Case 4:} If the spectral triple is 
$\mathcal{M}_{1,\infty}^{(p)}-$summable, 
then the proof of Theorem \ref{main result} \eqref{mainb} 
follows that of Theorem \ref{main result} \eqref{maina} 
(see Cases 1,2,3 above) {\it mutatis mutandi}.

\appendix
\section{Computation of coboundaries}\label{cobos}

\begin{compo} Let $\mathscr{A}\subset\{1,\dots,p\}$ be such that $m-1,m\in\mathscr{A}.$ Let $\varphi$ be a trace on $\mathcal{L}_{1,\infty}$ (respectively, on $\mathcal{M}_{1,\infty}$). The mapping on 
$\mathcal{A}^{\otimes (p+1)}$ defined by
$c\to\varphi(\mathcal{W}_{\mathscr{A}}(c)D^{-|\mathscr{A}|})$
is a Hochschild coboundary of the multilinear mapping
$$
\theta:a_0\otimes\cdots\otimes a_{p-1}\to\frac{(-1)^{m-1}}{2}\varphi\left(\Gamma a_0\prod_{k=1}^{m-2}[b_k,a_k]\delta^2(a_{m-1})\prod_{k=m}^{p-1}[b_{k+1},a_k]D^{-|\mathscr{A}|}\right).
$$
\end{compo}
\begin{proof} For brevity, we prove the assertion for $p=2$ as the proof in the general case is very similar. We have
\begin{align*}
(b\theta)(a_0,a_1,a_2)&=\theta(a_0a_1,a_2)-\theta(a_0,a_1a_2)+\theta(a_2a_0,a_1)\\
&=-\frac12\varphi(\Gamma a_0a_1\delta^2(a_2)|D|^{-2})+\frac12\varphi(\Gamma a_0\delta^2(a_1a_2)|D|^{-2})-\frac12\varphi(\Gamma a_2a_0\delta^2(a_1)|D|^{-2}).
\end{align*}
Since $\Gamma$ commutes with $a_2$ and since $\varphi$ is a trace, it follows that
$$
\varphi(\Gamma a_2a_0\delta^2(a_1)|D|^{-2})=\varphi(\Gamma a_0\delta^2(a_1)|D|^{-2}a_2)
=\varphi(\Gamma a_0\delta^2(a_1)a_2|D|^{-2})+\varphi(\Gamma a_0\delta^2(a_1)[|D|^{-2},a_2]).$$
We have
$$[|D|^{-2},a_2]=-|D|^{-1}\delta(a_2)|D|^{-2}-|D|^{-2}\delta(a_2)|D|^{-1}\in\mathcal{L}_{1/3,\infty}\subset\mathcal{L}_1.$$
Therefore,
$$\varphi(\Gamma a_2a_0\delta^2(a_1)|D|^{-2})=\varphi(\Gamma a_0\delta^2(a_1)a_2|D|^{-2}a_2).$$
Finally, we have
$$(b\theta)(a_0,a_1,a_2)=\frac12\varphi(\Gamma a_0(\delta^2(a_1a_2)-a_1\delta^2(a_2)-\delta^2(a_1)a_2)|D|^{-2})$$ and 
since
$\delta^2(a_1a_2)-a_1\delta^2(a_2)-\delta^2(a_1)a_2=2\delta(a_1)\delta(a_2),$
the assertion follows.
\end{proof}

\begin{compo} Let $\mathscr{A}_1,\mathscr{A}_2\subset\{1,\dots,p\}$ be such that $|\mathscr{A}_1|=|\mathscr{A}_2|$ and  $\mathscr{A}_1\Delta\mathscr{A}_2=\{m-1,m\}.$ Let $\varphi$ be a trace on $\mathcal{L}_{1,\infty}$ (respectively, on $\mathcal{M}_{1,\infty}$). The mapping
on $\mathcal{A}^{\otimes (p+1)}$ defined by
$$
c\to\varphi(\mathcal{W}_{\mathscr{A}_1}(c)D^{-|\mathscr{A}_1|})
+\varphi(\mathcal{W}_{\mathscr{A}_2}(c)D^{-|\mathscr{A}_2|})
$$
is a Hochschild coboundary of the multilinear mapping
$$
\theta:a_0\otimes\cdots\otimes a_{p-1}\to(-1)^{m-1}\varphi\left(\Gamma a_0\prod_{k=1}^{m-2}[b_k,a_k][F,\delta(a_{m-1})]\prod_{k=m}^{p-1}[b_{k+1},a_k]D^{-|\mathscr{A}_1|}\right).
$$
\end{compo}
\begin{proof} For brevity, we prove the assertion for $p=2$ as the proof in the general case is a slight extension of this argument. We have
$$
(b\theta)(a_0,a_1,a_2)=\theta(a_0a_1,a_2)-\theta(a_0,a_1a_2)+\theta(a_2a_0,a_1)
$$
$$
=-\varphi(\Gamma a_0a_1[F,\delta(a_2)]|D|^{-1})
+\varphi(\Gamma a_0[F,\delta(a_1a_2)]|D|^{-1})-\varphi(\Gamma a_2a_0[F,\delta(a_1)]|D|^{-1}).
$$
Since $\Gamma$ commutes with $a_2$ and since $\varphi$ is a trace, it follows that
\begin{align*}
\varphi(\Gamma a_2a_0[F,\delta(a_1)]|D|^{-1})
&=\varphi(\Gamma a_0[F,\delta(a_1)]|D|^{-1}a_2)\\
&=\varphi(\Gamma a_0[F,\delta(a_1)]a_2|D|^{-1})+\varphi(\Gamma a_0[F,\delta(a_1)][|D|^{-1},a_2]).
\end{align*}
We have
$$
[|D|^{-1},a_2]=-|D|^{-1}\delta(a_2)|D|^{-1}\in\mathcal{L}_{1/2,\infty}\subset\mathcal{L}_1.
$$
Therefore,
$$
\varphi(\Gamma a_2a_0[F,\delta(a_1)]|D|^{-1})=\varphi(\Gamma a_0[F,\delta(a_1)]a_2|D|^{-1}).
$$
Finally, we have
$$
(b\theta)(a_0,a_1,a_2)=\varphi(\Gamma a_0([F,\delta(a_1a_2)]-a_1[F,\delta(a_2)]-[F,\delta(a_1)]a_2)|D|^{-1}).
$$
Since
$$
[F,\delta(a_1a_2)]-a_1[F,\delta(a_2)]-[F,\delta(a_1)]a_2=[F,a_1]\delta(a_2)+\delta(a_1)[F,a_2],
$$
the assertion follows.
\end{proof}


\begin{thebibliography}{99}
\bibitem{BeF} M. T. Benameur and T. Fack {\it Type $II$ noncommutative geometry,
I. Dixmier trace in von Neumann algebras},  Adv. Math. {\bf 199} (2006), 29--87.
\bibitem{CGRS2} A. Carey, V. Gayral, A. Rennie, F. Sukochev, {\it Index theory for locally compact noncommutative geometries,} Mem. Amer. Math. Soc. to appear.
\bibitem{CPRS} A. Carey, J. Phillips, A. Rennie, F. Sukochev, {\it The Hochschild class of the Chern character for semifinite spectral triples},  J. Funct. Anal., {\bf 213}, (2004) no. 1, 111--153.
\bibitem{CPS1} A. Carey, J. Phillips, F. Sukochev, {\it On unbounded $p-$summable Fredholm modules,} Adv. Math. {\bf 151} (2000), no. 2, 140--163.
\bibitem{CRSS} A. Carey, A. Rennie, A. Sedaev, F. Sukochev, {\it The Dixmier trace and asymptotics of zeta functions,} J. Funct. Anal. {\bf 249} (2007), no. 2, 253--283.
\bibitem{CarSuk} A. Carey, F. Sukochev, {\it Measurable operators and the asymptotics of heat kernels and zeta functions,} J. Funct. Anal. {\bf 262} (2012), no. 10, 4582--4599.
\bibitem{Connes} A. Connes, {\it Noncommutative Geometry,} Academic Press, San Diego, 1994.
\bibitem{Dixmier} J. Dixmier, {\it Existence de traces non normales,} (French)  C. R. Acad. Sci. Paris Ser. A-B  {\bf 262}  (1966) A1107--A1108.
\bibitem{DFWW} K. Dykema, T. Figiel, G. Weiss, M. Wodzicki, {\it Commutator structure of operator ideals,}  Adv. Math.  {\bf 185} (2004), no. 1, 1--79.
\bibitem{GVF}
J. M. Gracia-Bond\'{i}a, J. C. V\'arilly, H. Figueroa,
\textit{Elements of Noncommutative Geometry},
Birkh\"auser, Boston, 2001.
\bibitem{Kalton1998} N. Kalton, {\it Spectral characterization of sums of commutators. I}, J. Reine Angew. Math. {\bf 504} (1998), 115--125.
\bibitem{KLPS} N. Kalton, S. Lord, D. Potapov, F. Sukochev, {\it Traces on compact operators and the noncommutative residue,} Adv. Math. {\bf 235} (2013), 1--55.
\bibitem{LT2} J. Lindenstrauss, L. Tzafriri, {\it Classical Banach spaces. II. Function spaces,} Ergebnisse der Mathematik und ihrer Grenzgebiete {\bf 97}. 
Springer-Verlag, Berlin-New York, 1979.
\bibitem{LSZ} S. Lord, F. Sukochev, D. Zanin, {\it Singular Traces: Theory and Applications,} volume 46 of Studies in Mathematics. De Gruyter, 2013.
\bibitem{PW}  S. Patnaik, G. Weiss, {\it Subideals of operators II,} Integral Equations Operator Theory {\bf 74} (2012), no. 4, 587--600.
\bibitem{PS} D. Potapov, F. Sukochev, {\it Unbounded Fredholm modules and double operator integrals,} J. Reine Angew. Math. {\bf 626} (2009), 159--185.
\bibitem{RenSmo} A. Rennie, {\it Smoothness and locality for nonunital spectral triples},
$K$-Theory \textbf{28} (2003), 127--165.
\bibitem{SSUZ-pietsch} E. Semenov, F. Sukochev, A. Usachev, D. Zanin, {\it Banach limits and traces on $\mathcal{L}_{1,\infty}$,} preprint.
\bibitem{SZ-JFA} F. Sukochev, D. Zanin, {\it $\zeta-$function and heat kernel formulae,} J. Funct. Anal. {\bf 260} (2011), no. 8, 2451--2482. 

\end{thebibliography}
\end{document}